\documentclass[11pt]{amsart}

\usepackage{amssymb,amsmath,amsthm}
\usepackage{color}
\usepackage{enumerate,hyperref,mathrsfs}

\DeclareMathOperator{\diag}{diag}
\DeclareMathOperator{\Int}{int}
\DeclareMathOperator{\RE}{Re}
\DeclareMathOperator{\range}{range}

\newcommand{\OS}[2]{\mathcal M_{{#2} \times {#1}}}

\def\hh{ h}

\def\cB{\mathcal B}
\def\Bgg{\cB_{g' \times g} }

\def\Bdd{\cB_{d' \times d} }

\def\Mdd{\CC^{d' \times d} }

\def\U{U}

\def\T{T}

\def\beq{\begin{equation} }
\def\eeq{\end{equation} }
\def\eps{\varepsilon}
\def\ep{\varepsilon}
\def\al{\alpha}
\def\be{\beta}
\def\Vx{x}
\def\VX{X}
\def\Vxy{\begin{bmatrix}x\\y\end{bmatrix}}

\def\db{{\rm dist}}

\def\w{w}
\def\ben{\begin{enumerate} }
\def\een{\end{enumerate} }
\def\benum{\begin{enumerate} }
\def\eenum{\end{enumerate} }

\def\x{x}
\newcommand{\ax}{\langle\x\rangle}
\def\oxs{\x^\T}
\newcommand{\axs}{\langle\x ,\oxs\rangle}
\def\bmat{\left[\begin{array}{ccccccccccccccc} }
\def\emat{\end{array}\right]}

\def\bmat{\begin{bmatrix}}
\def\emat{\end{bmatrix}}
\def\beq{\begin{equation}}
\def\eeq{\end{equation}}

\def\barr{\begin{array}}
\def\earr{\end{array}}

\def\cB{\mathcal{B}}

\def\T{\ast}

\def\NN{\mathbb N}

\def\cB{ {\mathcal B} }

\def\cD{ {\mathcal D} }

\def\cF{ {\mathscr F} }

\def\cN{ {\mathcal N} }

\def\cS{{\mathcal S} }

\def\cW{{\mathcal W}}
\def\cU{\mathcal U}

\def\tL{{\tilde L}}

\def\j1tog{ j= 1, \ldots, g }

\def\Bg{\cB_{g^\prime}}
\def\Bgg{ {\cB}_{g^{\prime}\times g} }
\def\Bdd{ {\cB}_{d^{\prime} \times d}}

\def\Fgg{\mathcal F_{g^\prime\times g}}
\def\Fggn{\Fgg(n)}

\def\Fgp{\mathcal F_{g^\prime}}
\def\Fg{\mathcal F_g}
\def\Fgpn{\Fgp(n)}
\def\Fgn{\Fg(n)}

\def\Mgg{\CC^{g^\prime\times g}}

\def\L0t{\ (L_0 \otimes I_n ) \ }

\def\PLINE1{\beta}

\newcommand{\CC}{{\mathbb C}}
\newcommand{\RR}{{\mathbb R}}
\newcommand{\DD}{{\mathbb D}}

\newcommand{\pt}{\partial}

\usepackage{fancyvrb}

\newtheorem{thm}{Theorem}[section]
\newtheorem{cor}[thm]{Corollary}
\newtheorem{lem}[thm]{Lemma}
\newtheorem{lemma}[thm]{Lemma}

\newtheorem{prop}[thm]{Proposition}
\theoremstyle{definition}

\newtheorem{defi}[thm]{Definition}

\theoremstyle{remark}
\newtheorem{rem}[thm]{Remark}

\numberwithin{equation}{section}

\newtheorem{exa}[thm]{Example}
         {\begin{exa}}
             {{\hfill $\Box ~~$}\end{exa}}

\newcounter{Inc}

\begin{document}
\setcounter{page}{1}

\title{Noncommutative ball maps}

\author[Helton]{J. William Helton${}^1$}
\address{Department of Mathematics\\
  University of California \\
  San Diego}
\email{helton@math.ucsd.edu}
\thanks{${}^1$Research supported by NSF grants 
DMS-0700758, DMS-0757212, and the Ford Motor Co.}

\author[Klep]{Igor Klep${}^2$}
\address{Univerza v Ljubljani, Oddelek za matematiko
In\v stituta za matematiko, fiziko in mehaniko 
}
\email{igor.klep@fmf.uni-lj.si}
\thanks{${}^2$Supported by the Slovenian Research Agency (project No. Z1-9570-0101-06).}

\author[McCullough]{Scott McCullough${}^3$}
\address{Department of Mathematics\\
  University of Florida 
   }
   \email{sam@math.ufl.edu}
\thanks{${}^3$Research supported by the NSF grant DMS-0758306.}

\author[Slinglend]{Nick Slinglend}
\address{Department of Mathematics\\
  University of California \\
  San Diego}
\email{nslingle@euclid.ucsd.edu}

\date{26 September 2008}


\begin{abstract}
In this paper,
we analyze problems involving matrix variables
for which we use a noncommutative algebra setting.
 To be more specific, we use a class of functions (called NC analytic
 functions)
 defined by
power series in noncommuting variables and
evaluate these functions on sets of matrices of all dimensions;
 we call such situations  dimension-free.
These types of functions have recently been used
 in the study of dimension-free linear
system engineering problems \cite{HMPVieee}, \cite{deOHMPima}.

In this paper we characterize NC analytic  maps that send
dimension-free matrix balls to dimension-free matrix
balls  and carry the boundary to the boundary;
such maps we call "NC ball maps".
We find that
up to normalization, an NC ball map
is the direct sum of the identity map with an
NC analytic map of the ball into the ball.
That is,
 ``NC ball maps'' are very simple, in contrast
to the classical result of D'Angelo on such analytic maps
in $\CC$.
Another mathematically natural class
of maps carries a variant of the
noncommutative distinguished boundary to the 
boundary,
but on these our results are  limited.

We shall be interested in several types of noncommutative
balls, conventional ones, but also balls defined by
constraints called Linear Matrix Inequalities (LMI).
What we do here is a small piece of the bigger puzzle
of understanding how LMIs behave with respect
to noncommutative change of variables.
\end{abstract}

\maketitle

\section{Introduction}

In the introduction we will state some of our main results.
For this we need to start with the definitions of NC
polynomials (\S \ref{subsec:NCpoly}) and
NC analytic maps (\S \ref{subsec:anal}).
We then proceed to define NC ball maps in \S \ref{subsec:main}, 
where we explain what it means for an NC ball map to map ball to ball with
boundary to boundary. 
After that we can and do state our main results classifying
NC ball maps in \S \ref{subsec:main}
and \S \ref{subsec:nonzero}. Finally, the introduction concludes by 
considering two types of generalizations, 
the first being to balls defined by LMIs, the second being
to NC analytic maps carrying special sets on the boundary of a ball to
the boundary of a ball.

\subsection{Words and NC polynomials}\label{subsec:NCpoly}
Let $g',g\in\NN$.
We write $\ax$ for the monoid freely
generated by $\x$, i.e., $\ax$ consists of {\bf words} in the $g'g$
letters $x_{11},\ldots,x_{1g},x_{21},\ldots,x_{g'g}$ 
(including the empty word $\emptyset$ which plays the role of the identity $1$).
Let $\CC\ax$ denote the associative
$\CC$-algebra freely generated by $\x$, i.e., the elements of $\CC\ax$
are polynomials in the noncommuting variables $\x$ with coefficients
in $\CC$. Its elements are called {\bf NC polynomials}.
An element of the form $aw$ where $0\neq a\in \CC$ and
$w\in\ax$ is called a {\bf monomial} and $a$ its
{\bf coefficient}. Hence words are monomials whose coefficient is
$1$.
Let $\oxs=(x_{11}^\T,\ldots,x_{g'g}^\T)$ denote another set of $g'g$ symbols.
We shall also consider
the free algebra $\CC\axs$ that comes equipped with the natural {\bf involution}
$x_{ij}\mapsto x_{ij}^\T$.
For example, $$(1+  i\, x_{11}^2 x_{23}^\T x_{34}^\T)^\T =1-  i\, x_{34} x_{23} (x_{11}^\T )^2.$$
(Here $ i$ denotes the imaginary unit $\sqrt{-1}$.)

\subsubsection{NC matrix polynomials}

A {\bf matrix valued NC polynomial} is an NC polynomial with matrix coefficients.
We shall use the phrase \emph{scalar} NC polynomial if we want to emphasize
the absence of matrix constructions. Often when the context makes the usage clear we drop
adjectives such as scalar, $1 \times 1$, matrix polynomial, matrix of polynomials
and the like.

\subsubsection{Polynomial evaluations}

If $p$ is an NC polynomial in $x$ and $X\in\left(\CC^{n\times n}\right)^{g'\times g}$,
the evaluation $p(X)$ is defined by simply replacing $x_{ij}$ by $X_{ij}$.
For example, if $p(x)=A x_{11} x_{21}$, where
$$
A=\left[\begin{array}{ccc}
-4&3&2\\
2&-1&0\\
\end{array}\right],
$$
then
$$
p\left( \left[\begin{array}{cc}
0&1 \\
1 & 0
\end{array}\right],
\left[\begin{array}{cc}
1&0\\
0&-1
\end{array}\right]
\right)=
A\otimes \left( \left[\begin{array}{cc}
0&1 \\
1 & 0
\end{array}\right]\,
\left[\begin{array}{cc}
1&0\\
0&-1
\end{array}\right]\right)=
\left[\begin{array}{cccccccccccc}
0 &4 &0 & -3 & 0 &-2\\
-4 & 0& 3 & 0 & 2 & 0\\
0&-2&0&1&0&0\\
2&0&-1&0&0&0
\end{array}\right].
$$
On the other hand, if $p(x)=A$ and $X\in\left(\CC^{n\times n}\right)^{g'\times g}$,
then $p(X)=A\otimes I_n$.

The tensor product in the expressions above is the usual (Kronecker) tensor product
of matrices. Thus we have reserved the tensor product notation for the tensor product
of matrices and have eschewed the strong temptation of using $A\otimes x_{k\ell}$ in place of
$Ax_{k\ell}$ when $x_{k\ell}$ is one of the noncommuting variables.

\subsection{Definition of NC analytic functions}
\label{subsec:anal}

  An elegant theory of noncommutative analytic functions is
  developed in the articles \cite{KV,KVV1,KVV} and \cite{Vo1,Vo}; see also \cite{Pop1}.
  What we need in this article are specializations of definitions
  of these papers.
  In this section we 
  summarize the definitions and properties
  needed in the sequel.

For $d',d\in\NN$  define
\begin{eqnarray}\label{eq:defBall3}
\cB_{d'\times d}& :=&
\bigcup_{n=1}^\infty\left\{ X\in\left(\CC^{n\times n}\right)^{d'\times d}
\mid  I_{dn}- X^\T X \succeq 0\right\}, \\
\Int\cB_{d'\times d}& :=&
\bigcup_{n=1}^\infty\left\{ X\in\left(\CC^{n\times n}\right)^{d'\times d}
\mid  I_{dn}- X^\T X \succ 0\right\}, \\
\partial\cB_{d'\times d}& :=&
\bigcup_{n=1}^\infty\left\{ X\in\left(\CC^{n\times n}\right)^{d'\times d}
\mid  \|X\|=1 \right\},  \\
\OS{d}{d'}& :=&
\bigcup_{n=1}^\infty \left(\CC^{n\times n}\right)^{d'\times d}.
\end{eqnarray}

We shall occasionally use the notation
\begin{equation*}
\begin{split}
  \Bdd(N)
& =\{X=
  \begin{bmatrix} X_{j,\ell}
  \end{bmatrix}_{j,\ell=1}^{d^\prime,d}\mid
   X_{j,\ell} \in \CC^{N\times N}, \; \|X\|\le 1\},\\
 \OS{d'}d(N)
& =\{X=
  \begin{bmatrix} X_{j,\ell}
  \end{bmatrix}_{j,\ell=1}^{d^\prime,d}\mid
   X_{j,\ell} \in \CC^{N\times N}\}.
\end{split}
 \end{equation*}

  Given $g',g\in\NN$, the 
  {\bf noncommutative (NC) $\eps$-neighborhood} of $0$
  in $\CC^{g'\times g}$
  is the (disjoint) union 
  $\bigcup_{N\in\NN}\{ X \in \OS g{g'}(N) \mid \|X\|<\eps \}$.
  An {\bf open NC domain} $\cD$ containing $0$ (in its interior)
  is a union $\bigcup_N \cD_N$ of open sets
   $\cD_N \subseteq \OS g{g'}(N)$ which is closed with
   respect to direct sums and such that there is an $\eps>0$
   such that $\cD$ contains the NC $\eps$-neighborhood of $0$.

A $d'\times d$ {\bf NC analytic function} $f$ on an open 
NC domain $\cD$ containing $0$
as follows:
\begin{enumerate}
\item
\label{it:series}
  $f$ has an {\bf NC power series}, for which there exists an
     NC $\eps>0$ neighborhood of $0$  on which it is convergent.
That is,
\beq\label{eq:powerSeriesDef}
  f=\sum_{w\in\ax} a_w w
\eeq
for $a_w\in\CC^{d'\times d}$ and
for every $N\in\NN$ and every $g'\times g$-tuple of square matrices
$X\in\Bgg$ with $\|X\|<\ep$ the series
\beq\label{eq:powerSeriesDefEval}
f(X)=\sum_{w\in\ax} a_w \otimes w(X)
\eeq
converges.
  We interpret convergence for a given $X$ as conditional
  of the series
 \begin{equation*}
    \sum_{\alpha=0}^\infty \sum_{|w|=\alpha} a_w \otimes w(X).
 \end{equation*}
   Thus the {\it order} of summation is over the
   homogeneous parts of the power series expansion.
   Thus with $f^{(\alpha)}$ equal to
  the $\alpha$ {\bf homogeneous part} in the NC power series expansion
  of $f$, the series converges for a given $X$ provided
 \begin{equation}
  \label{eq:order}
   \sum_{\alpha=0}^\infty f^{(\alpha)}(X)
 \end{equation}
  converges.
   Since both $a_w$ and $w(X)$ are matrices, the
   particular norm topology chosen has no influence
   on convergence.
The radius $\ep$ of this ball of convergence
(or sometimes, by abuse of notation, the ball itself)
will be called the {\bf series radius}.

\item
\label{it:opsub}
  If $a : \cW \to \cD$ is a matrix valued function analytic
      on a domain $\cW$ in $\CC^N$, the composition
      $ f \circ a$ is a matrix valued analytic function
      on $\cW$ and continuous to $\pt\cW$.
\end{enumerate}
\noindent

\begin{rem}
  Popescu \cite{Pop1} has a notion of free analytic function in $g'$ variables
  (that is, $g=1$) based
  upon power series expansions like that in \eqref{eq:order}. His
  definition allows for operator coefficients, but on the other hand
  requires convergence of the NC power series on all of $\Int \cB_{g'}$
  (the NC $1$-neighborhood of $0$ in $\CC^{g'}$).
  It turns out that for \emph{bounded} NC analytic functions with
  \emph{matrix coefficients} the two notions
  are the same, see Lemma \ref{lem:abs-converge-contractive}.
\end{rem}

\subsubsection{Properties of NC analytic functions}

\begin{prop}\label{prop:anal}
Let $\cD$ be an NC domain containing $0$.
\ben[\rm (i)]
\item
The sum of two $d'\times d$ NC analytic functions on
$\cD$ is a $d'\times d$ NC analytic function on $\cD$.
\item
The product of two $d'\times d$ NC analytic functions on
$\cD$ is a $d'\times d$ NC analytic function on $\cD$.
\item
The composition of two NC analytic functions
is
an NC analytic function. More precisely, if
$f:\cD\to\cD '$ is a $d_1'\times d_1$ NC analytic function,
where $\cD '$ is an NC domain with $0\in\cD '$,
and $h$ is a $d_2'\times d_2$ NC analytic function on $\cD '$,
then $h\circ f$ is a $d_1'd_2'\times d_1 d_2$
NC analytic function on
$\cD$.
\een
\end{prop}

\begin{proof}
Properties (i) and (ii) are standard and we only consider
(iii). The fact that $h\circ f$ admits an NC power series
as in
\eqref{eq:powerSeriesDef} was observed e.g. in 
\cite{KVV,Vo1}.
The composition property \eqref{it:opsub} of \S
\ref{subsec:anal} is easily checked.
\end{proof}

More is said about properties of NC analytic functions in
\S \ref{sec:NCanalAgain}.

\subsection{NC ball maps $f$ and their classification when $f(0)=0$}
 \label{subsec:main}
  A function $$f:\Int\Bgg \to\OS{d'}d$$ which is NC analytic
  will often be  called
 an {\bf NC analytic function on the ball $\Bgg$} and denoted
 $f:\Bgg \to\OS{d'}d$.  
  An NC analytic function
  $f : \Bgg \to\cB_{d'\times d}$ mapping the boundary to the
  boundary is called an {\bf NC ball map}.
 The notion of  $f$ mapping boundary to boundary is a bit
 complicated (because of convergence issues)
 so requires explanation.
  For a given $X\in\Bgg(N)$, define the function 
  $f_X:\mathbb D \to \OS{d'}d(N)$  by
  $z \mapsto f(zX)$.
  (Here $\mathbb D$ denotes the unit disc 
$\mathbb D=\{z\in\mathbb C \mid |z|<1\}$ in the complex plane.)
  If  $$\lim_{r \nearrow 1} f_X(r e^{it})$$ exists,  denote that limit by $f(e^{it}X)$.  
  The function $f$ {\bf maps the boundary to the boundary} if 
  whenever
  $\|X\|=1$ and $f(e^{it}X)$ exists, then 
$$
\| f(e^{it}X) \| = 1.
$$ 
Since $f$ is bounded,
Fatou's Theorem implies that for  each $X \in \Bgg$ the limit 
  $f_X(e^{it})=f(e^{it}X)$ exists for almost every $t$. 
  If $f$ is an NC ball map, $X\in\partial \Bgg$
  and $f(X)$ is defined, then a 
  (nonzero) vector $\gamma$ such that
  $\|f(X)\gamma\|=\|\gamma\|$ is called
   a {\bf binding vector} and this property
  {\bf binding}. 

Our main result on NC ball maps which map $0$ to $0$ is:

\begin{thm}
 \label{thm:ball-NO-L}
 Let $h:\Bgg \to\Bdd$ be
 an NC ball map with $h(0)=0$.
 Then
  there exist unitaries
  $U:\CC^{d} \to \CC^{d }$
  and $V:\CC^{d'}\to\CC^{d'}$ such that
  \beq
  \label{eq:ballDeal}
   h(x) =V \left[\begin{array}{cc} 
   x&0\\
   0&\tilde h(x)
   \end{array}\right] U^\T ,
   \eeq
   where $\tilde h:\Bgg\to\cB_{(d'-g')\times (d-g)}$ is an NC analytic contraction-valued map with
   $\tilde h(0)=0$.

   Conversely, every NC analytic
   $h$ satisfying
   \eqref{eq:ballDeal} for unitaries $U,V$
   and an NC analytic contraction-valued map $\tilde h$ fixing the origin,
   is an NC ball map $\Bgg \to \cB_{d'\times d}$
   sending $0$ to $0$.
   \end{thm}

   The proof of the theorem is completed in \S \ref{sec:scottsPf}.
As an illustration of Theorem \ref{thm:ball-NO-L} we describe a special case.
For convenience, we adopt the notation $\cB_{g^\prime}$ for
  $\cB_{g^\prime\times 1}$. 

\begin{cor}\label{cor:bigDeal}
If $h: \cB_{g'}\to\cB_{d'}$ is an NC ball map with $h(0)=0$,
then $h$
is linear and there is a unique isometry $M\in\CC^{d'\times g'}$ such that $h=M\Vx $.
In particular, if $d'<g'$ then no such NC ball maps exist.
\end{cor}

\begin{proof}
When $h$ maps $\cB_{g'}$ to $\cB_{d'}$ then the $\tilde h(x)$ column is gone. 
Moreover, 
$$
M=V^\T\begin{bmatrix}I\\0\end{bmatrix}
$$
is an isometry.
\end{proof}

\subsection{NC Ball maps $f$ when $f(0)$ is not necessarily $0$}\label{subsec:nonzero}
  In the previous section we treated NC ball maps
  $f$ with $f(0)=0$, an assumption we drop in
  this section. The strategy is to compose $f$
  with a bianalytic automorphism of an NC ball
  to reduce the problem to the $f(0)=0$ setting. 
  \S \ref{sec:linearfrac} contains information
  on bianalytic mappings on a NC ball, while
  the main results appear in \S \ref{sec:classify}.

\subsubsection{Linear fractional transformations}
 \label{sec:linearfrac}
For a given $d^\prime \times d$ scalar matrix $v$ with $\|v\|<1$, define
$\cF_v:\Bdd\to \Bdd$ by 
\beq\label{eq:deffIntro}
 \cF_v (u):=v-(I_{d'}-vv^\T )^{1/2}u(I_d-v^\T u)^{-1}(I_d-v^\T v)^{1/2}.
\eeq
 Of course it must be shown that $\cF_v$ actually takes values
 in $\Bdd$. This is done in Lemma \ref{lemma:multilinfracIntro}
 below.

 Linear fractional transformations such 
 as $\cF_v$ are common in circuit and system theory,
since they are associated with energy conserving pieces of a circuit (cf.
\cite{Wo})

\begin{lem}\label{lem:linFrac1Intro}
 Suppose $\cD$ is an open
NC domain containing $0$. 
  If
  $u:\mathcal D\to \Bdd$ is NC analytic, then 
 $\cF_v (u(x))$ is an NC analytic function $($in $x)$ on $\cD$.
\end{lem}

\begin{proof}
See \S \ref{sec:lfl}.
\end{proof}

Notice that if $d=d'=1$, then $v$ is a scalar and $u$ is a scalar NC analytic
function, hence
$$ \cF_v (u)= (v-u)(1-u\bar v)^{-1}
= (1-u\bar v)^{-1}(v-u).$$
Now fix $v\in\DD$ and consider the map $\DD\to\CC$,
$
u\mapsto \cF_v (u).
$ 
This map is a
linear fractional map that maps the unit disc to the unit disc, maps
the unit circle to the unit circle, and maps $v$ to 0.

The geometric interpretation of the map in NC variables in
(\ref{eq:deffIntro}) is similar. Suppose  we fix $N\in \mathbb{N}$ and
$V\in  \Bdd(N)$ with $\| V\| <1$ and consider the map
\beq
\label{eq:lfIntro} U \mapsto \cF_V (U).
\eeq
The first part of Lemma
\ref{lemma:multilinfracIntro} tells us that the map defined in
\eqref{eq:lfIntro} maps the unit ball of $d'\times d$-tuples of $N \times
N$ matrices to the unit ball of $d'\times d$-tuples of $N\times N$ matrices
carrying the boundary to the boundary.  The third part of Lemma
\ref{lemma:multilinfracIntro} tells us that $\cF_V (V)=0$; that is, the
 map given in \eqref{eq:lfIntro} takes $V$ to 0.

\begin{lemma} \label{lemma:multilinfracIntro}
Suppose that  $N\in \mathbb{N}$ and $V\in\Bdd(N)$ with $\| V\|
<1$.
\begin{enumerate}[\rm (1)]
\item
$U\mapsto \cF_V(U)$ maps the $\Bdd(N)$ into itself
with boundary to the boundary.
\item If $U\in\Bdd(N)$, then
$\cF_V(\cF_V(U))=U.$
\item $\cF_V(V)=0$ and $\cF_V(0)=V$.
\end{enumerate}
\end{lemma}

\begin{proof}
See \S \ref{sec:lfl}.
\end{proof}

\subsubsection{Classification of NC Ball maps}
 \label{sec:classify}
   General NC ball maps - those where $f(0)$ is not necessarily
   $0$ - are described using the linear fractional transformation $\cF$.

\begin{thm}
 \label{thm:ball-NO-L-NO-0}
 Let $f:\Bgg \to\Bdd$ be
 an NC ball map with $f(0)\not\in\partial\Bdd$.
 Then
\beq\label{eq:bigDeal0fla}
f(x)= \cF_{f(0)}\big(\varphi(x) \big),
\eeq
where 
\beq\label{eq:bigDeal1fla}
\varphi(x)=\cF_{f(0)}\big(f(x)\big)=
V
\left[\begin{array}{cc} 
   x&0\\
   0&\tilde\varphi(x)
   \end{array}\right] U^\T 
\eeq
  for some unitaries
  $U:\CC^{d} \to \CC^{d }$
  and $V:\CC^{d'}\to\CC^{d'}$ and 
  an NC analytic contraction-valued map $\tilde\varphi$ with
  $\|\tilde\varphi(0)\|<1$.

   Conversely, every NC analytic
   $f$ satisfying
   \eqref{eq:bigDeal0fla} and \eqref{eq:bigDeal1fla} for unitaries $U,V$
   and $\tilde\varphi$ as above,
   is an NC ball map $f:\Bgg \to \cB_{d'\times d}$ with 
   $f(0)\not\in\partial\Bdd$.
\end{thm}

\begin{proof}
Define $\varphi(x):=\cF_{f(0)}\big(f(x)\big)$. Then $\varphi(0)=0$.
By Lemma \ref{lem:linFrac1Intro}, $\varphi(x)$ is an NC analytic map.
Hence it is an NC ball map sending $0$ to $0$ and is thus classified 
by Theorem \ref{thm:ball-NO-L}. Moreover,
the equation \eqref{eq:bigDeal0fla} is implied by
Lemma \ref{lemma:multilinfracIntro}.(2). The converse easily follows from
Lemmas \ref{lem:linFrac1Intro} and \ref{lemma:multilinfracIntro}.
\end{proof}

The results of \S \ref{subsec:main} and \S \ref{subsec:nonzero} 
are treated in 
Part I of this paper. 

\subsection{More generality}

In this subsection we extend the main results presented so far
in two directions. The first concerns LMIs.
Our interest will be in properties of the set of all solutions
to a given  LMI.
In \S \ref{subsubsec:lmi} we will define what we mean by
an LMI,
then show that the set of solutions to a ``monic'' LMI
equals a general type of matrix ball we call a pencil ball.
Ultimately we would like to study maps from pencil balls to pencil
balls and
this
paper is a beginning
which
handles the special case where
the domain pencil ball is the ordinary NC ball $\Bgg$ (see Corollary
\ref{cor:bigDeal2}).
Eventually we hope
to understand which
 NC analytic change of variables takes one LMI to another.
Work is in progress on such problems.

In the next generalization we do not have applications in mind,
but do something that is mathematically natural.
A basic notion in several complex variables is the Shilov or distinguished
boundary. A natural problem is to classify NC analytic functions
mapping the ball to the ball and carrying the distinguished boundary to 
the boundary. Classification of linear maps of this type proves
to be an interesting challenge tackled in \S \ref{sec:linear-bininding} 
and \S \ref{sec:further}.
For NC analytic maps we introduce the semi-distinguished boundary (a set larger
than the distinguished boundary) and study the NC analytic
functions mapping the semi-distinguished boundary to the boundary.
All of this we only do for balls of vectors, rather than
balls of matrices, that is for $g=1$.

\subsubsection{Linear pencils}
Let
\beq
\label{eq:defLP}
L(x):=A_{11}x_{11}+\cdots+A_{g'g} x_{g'g}
\eeq
denote an {\bf NC analytic truly linear pencil}
in $x$. If the matrices
$A_{ij}$ that are used to define it are in $\CC^{d'\times d}$, then
$L(x)$ is called a $d'\times d$ linear pencil.
As an example, for $g'=2$ and $g=1$,
$$
A_{11}=\left[\begin{array}{cc}
1 & 2\\
3 & 4
\end{array}\right],\quad
A_{21}=\left[\begin{array}{cc}
0 & 1\\
-1 & 0
\end{array}\right],
$$
the linear pencil is
$$
L(x)=\left[\begin{array}{cc}
x_{11} & 2x_{11}+x_{21}\\
3x_{11}-x_{21}& 4 x_{11}
\end{array}\right].
$$

\subsubsection{Linear matrix inequalities and {\rm (}pencil{\rm )} balls}\label{subsubsec:lmi}

Let $\tL$ be a $d\times d$ monic symmetric linear pencil.
The positivity domain of $\tL$ is defined to be
$$
\cD_{\tL} := \{X\in\OS g{g'} \mid \tL (X)\succeq 0 \}.
$$
In other words, it is the set of all solutions to
the LMI $\tL (X)\succeq 0$.
We wish to analyze this solution set and we can using
results on balls which we have already obtained.
Now we describe $\cD_\tL$ as a type of ball.
To do this write $\tL$ as $\tL = I+L+L^\T$
where
$L$ is a $d\times d$ NC analytic truly linear pencil,
then
to $L(x)$ we associate the {\bf (pencil) ball}
\beq\label{eq:defBall1}
\cB_{L}:=\bigcup_{n=1}^\infty\left\{X \in \OS g{g'}(n)
\mid I_{dn}- L(X)^\T L(X) \succeq 0\right\}
=\bigcup_{n=1}^\infty
\left\{X\in \OS g{g'}(n) \mid \|L(X)\|\leq 1\right\}.
\eeq
Observe that $\Bgg= \cB_L$ for
$$L(x) = \sum_{i,j} E_{ij} x_{ij}$$
with $E_{ij}$ being the elementary $g'\times g$ matrix with $1$ located
at position $(i,j)$.

\begin{lem}\label{lem:dl}
For $X\in\OS g{g'}$,
\beq
\left[\begin{array}{cc}
0& X\\
0&0
\end{array}
\right]
\in\cD_{\tL} \quad \text{ iff }\quad X\in  \cB_{L}.
\eeq
Furthermore,
\beq
\left[\begin{array}{cc}
0&X\\0&0\end{array}\right]\in\pt\cD_{\tL} \quad \text{ iff }\quad
X\in\pt \cB_{L}.
\eeq
\end{lem}

\begin{proof}
By definition,
$$\tL\left(\left[\begin{array}{cc}0&X\\0&0\end{array}\right]\right)
=\left[\begin{array}{cc}I&L(X)\\L(X)^\T &I\end{array}\right]
.
$$
\end{proof}

\subsubsection{Pencil ball maps}

Now we turn to 
$\Bgg\to\cB_L$ maps.
  As a generalization of NC ball map, given
  a linear pencil $L$, an NC analytic
  mapping $f:\Bgg \to \cB_L$ will be called a {\bf pencil ball map}
  provided $\|L(f(X))\|= 1$, whenever 
  $\|X\|=1$ and $f(X)$ is defined. 
Lemma \ref{lem:dl} tells us that understanding 
pencil ball maps is equivalent to understanding maps on the sets
of solutions to certain types of LMIs.
  
\begin{thm}
 \label{thm:ball}
Let $L$ be a
$d'\times d$ NC analytic truly linear pencil and
$f:\Bgg \to\cB_{L}$
a pencil ball map with $f(0)=0$.
Write $h:=L\circ f$.
Then
 there exist unitaries
$U:\CC^{d} \to \CC^{d }$
and $V:\CC^{d'}\to\CC^{d'}$ such that
\beq
\label{eq:ballDealCOR}
 h(x) =V \left[\begin{array}{cc} 
x&0\\
0&\tilde h(x)
\end{array}\right] U^\T ,
\eeq
where $\tilde h$ is an NC analytic contraction-valued map.
\end{thm}

\begin{proof}
Follows easily by applying Theorem \ref{thm:ball-NO-L} to
$h$.
\end{proof}

\begin{cor}\label{cor:bigDeal2}
Let $L$ be a $d'\times d$ NC analytic truly linear 
pencil and $f:
\Bgg\to\cB_{L}$ a pencil ball map with 
$\|L\circ f(0)\|<1$.
Then
\beq\label{eq:bigDeal2}
L\circ f(x)= \cF_{L\circ f(0)}\big(\varphi(x) \big),
\eeq
where $\varphi(x)=\cF_{L \circ f(0)}\big(L\circ f(x)\big)$ is an NC ball map
$\Bgg\to\cB_{d'\times d}$
taking $0$ to
$0$ and is therefore completely described by Theorem {\rm \ref{thm:ball-NO-L}}.
\end{cor}

\begin{proof}
Apply Theorem \ref{thm:ball-NO-L-NO-0} to $L\circ f(x)$.
\end{proof}

\subsubsection{Semi-distinguished pencil ball maps}
 \label{subsec:distinguished}
Many of our proofs with little extra effort work for a class of functions
more general than pencil ball maps. These involve the notion of distinguished
boundary which we now define.

The Shilov boundary or {\bf distinguished boundary} of
$\Bgg(N)$ is the smallest closed subset $\Delta$ of
$\Bgg(N)$ with the following property: 
For $f: \Bgg(N) \to \CC^K$ analytic and continuous
to the boundary $\partial \Bgg(N)$,
for any $X \in \Bgg(N) $ we have
\beq
\|f(X) \| \leq \max_{U \in \Delta} \| f(U) \|.
\eeq
In other words, the maximum of $f$ over  $\Bgg(N)$ occurs in
the distinguished boundary. We refer the reader to
\cite[p. 145]{Kr} or \cite[Ch. 4]{He} for more details.

  It is a Theorem \cite[p. 77]{Ab} that the distinguished boundary
  of $\cB_{g}(N)$ is $$\{X\in\cB_g(N)\mid X^\T X=I\}.$$ Accordingly,
  we let $\partial_\db \cB_g$ denote the
  disjoint union of these distinguished boundaries
  and call this the distinguished boundary of $\cB_g$.  
  A further discussion of distinguished boundaries
  for $\Bgg$ is in \S \ref{subsec:distinguished-boundary}.

An NC analytic function $f:\cB_{g}\to\cB_{L}$
satisfying $f(0)\not\in \pt\cB_{L}$
and
\beq\label{eq:defDistMap}
f\left( \pt_\db \cB_{g} \right) \subseteq \pt\cB_{L}
\eeq
is called a {\bf distinguished pencil ball map}.
Here, \eqref{eq:defDistMap} means that for every isometry $X$
for which $\lim\limits_{\delta\nearrow 1} f(\delta X)$ exists,
this limit  lies in $\pt\cB_L$.

A natural open question is: classify distinguished pencil ball maps.
Our proof of Theorem \ref{thm:bidisk}
does something like this but a little weaker.
A key distinction between the semi-distinguished 
maps and the case treated
earlier in Theorems \ref{thm:ball-NO-L} and \ref{thm:ball}
occurs with \emph{linear} distinguished ball maps.
These we find much harder to classify than linear NC ball maps,
which we leave as an interesting open question.

\begin{defi}
The {\bf semi-distinguished boundary} of $\cB_{g'}$ is
defined to be
$$
\partial_{\db}^{1/2} \cB_{g'}:=
\bigcup_{n=1}^\infty \left\{ X\in \cB_{g'}(n)  \mid
X^\T X \text{ is a projection of dimension } \geq  \frac 12n \right\}.
$$
An NC analytic function $f:\cB_{g'}\to\cB_{L}$
satisfying
$f(0)\not\in \pt\cB_{L}$ and
\beq\label{eq:defSemiDistMap}
f\left( \pt_\db^{1/2} \cB_{g'} \right) \subseteq \pt\cB_{L}
\eeq
is called a {\bf semi-distinguished pencil ball map}.
Here, \eqref{eq:defSemiDistMap} means that for every $X\in \pt_\db^{1/2} \cB_{g'}$
for which $\lim\limits_{\delta\nearrow 1} f(\delta X)$ exists,
this limit  lies in $\pt\cB_L$.
\end{defi}

The study of semi-distinguished pencil ball
maps is the subject of Part II of this article. 
For semi-distinguished pencil ball maps we get a weak version
of the pencil ball map classification Theorem \ref{thm:ball} --
see Theorem \ref{thm:bidisk}.


\part*{\centerline{Part I. Binding}}
 \label{partI}
\section{Models for NC contractions}
 \label{sec:model}
  Let $S$ denote the ($g'$-tuple of) shift(s) on {\bf noncommutative
  Fock space} $\mathcal F_{g'}$. The Hilbert space $\mathcal F_{g'}$
  is the Hilbert space with orthonormal basis
  consisting of words $\ax$ in $g'$ NC variables
  $x=(x_1,\dots,x_{g'})$. Then $S_j w=x_jw$ for a word $w\in\ax$ and
  $S_j$ extends by linearity and continuity to $\mathcal F_{g'}$.
  The key properties we need about $S$ are:
 \begin{equation}
  \label{eq:key-thing}
  \begin{split}
    S_j^{\T} S_\ell & =  \delta_{j}^{\ell} I \qquad \text{for  } j,\ell=1,\ldots, g' \\
    I-\sum_{j=1}^{g'} S_j S_j^\T  & =  P_0,
  \end{split}
 \end{equation}
  where $P_0$ is the (rank one) projection onto the span of
  the empty word.

  A {\bf column contraction} is a $g'$-tuple of square matrices (operators),
 \begin{equation*}
   X=\begin{bmatrix} X_1 \\ \vdots \\ X_{g'} \end{bmatrix}
 \end{equation*}
  such that $I-X^\T X=I-\sum X_j^\T  X_j \succeq 0$.
If $X$ acts on finite dimensional space, then
$X$ is a column contraction if and only if $X\in\cB_{g'}$. In general, $X$
is a column contraction if and only if $X^\T $ is a row
  contraction. Row contractions (and so column contractions too)
  are well studied -- e.g.~by Popescu and also Arveson.
 A {\bf strict column
  contraction} is a column contraction $X$ for which there
  is an $\eps >0$ such that
  $I-\sum X_j^\T  X_j \succeq \eps.$ If $X$ is acting on a finite
  dimensional space, this last condition is equivalent to
  $I-\sum X_j^\T  X_j \succ 0$, i.e., $X\in\Int\cB_{g'}$.
  Column contractions are modeled by $S^\T $, which is
  the content of Lemma \ref{lem:model}
  below and a major motivation for these definitions. We do not
use this property of the $S_j$ until proving Theorem \ref{thm:schwarz-lemma}.

 \begin{lem}[\cite{F},\cite{Pop0}]
  \label{lem:model}
    If $X$ is a strict column contraction acting on a Hilbert space
   $\mathcal H$, then there is a Hilbert space $\mathcal K$ and an
   isometry $V:\mathcal H \to \mathcal K \otimes  \mathcal F_{g'}$
   such that $VX = (I \otimes S^\T) V$; i.e., for each $j$,
   $VX_j = (I\otimes  S_j^\T)V$ and in particular, for each
   word $w\in\ax$, $V w(X) = (I\otimes w(S^\T)) V$.
    Here $I$ is the identity on $\mathcal K$. Further, if $X\in\cB_{g'}(N)$
   $($so is a tuple of matrices$)$, then the dimension of $\mathcal K$
    can be assumed to be at most $N$.
\end{lem}

  A natural generalization of the $g^\prime$-tuple of shifts on
  Fock space to the $\OS{g}{g^\prime}$ and its (sequence of) ball(s)
  is 
$$
 \mathbb X =\begin{bmatrix} S_j^{\T} \otimes S_\ell \end{bmatrix}_{j,\ell=1}^{g^\prime,g}
$$
  (A word of {\it caution}: we have abused notation by using $S_j$ to denote
  shifts on both $\Fgp$ and $\Fg$.)  The operator $\mathbb X$
  should be compared to the {\it reconstruction operator} in \cite{Pop4}.

  Though we do not know if $\mathbb X$ serves as a universal model 
  for $\Bgg$ in the same way that $S$ does for $\Bg$,
  it does serve as a type of boundary for NC analytic functions.  
  The statement of the results requires approximating $\mathbb X$
  by matrices. The operator (not matrix) $\mathbb X$
  acts upon $\Fgg$ -- the Hilbert space with orthonormal basis
  consisting of words in $g'g$ NC variables
  $x=(x_{j,\ell})_{j,\ell =1}^{g',g}$.
 Given a natural number $n$, let 
  $\Fgn$ denote the span of words of length at most $n$ in $\Fg$, 
  and set $\Fggn =\Fgpn\otimes \Fgn$.  Let $\mathbb X_n$
  denote the compression of $\mathbb X$ to the (semi-invariant 
finite dimensional) subspace
  $\Fggn$.

\begin{lemma}
 \label{lem:funnyXn}
    Let $P_n$ denote the projection onto the complement of the
    span of $\emptyset$ in $\Fgpn$ $($and also in $\Fgn )$
    and let $Q_n$ denote the projection onto the complement
    of the span of $\{w \mid w \text{ is a word of length } n \}$
    in $\Fgpn$ $($and also in $\Fgn )$. Then:
\begin{equation}
  \label{eq:key-funny}
 \begin{split}
  \mathbb X_n^{\T} \mathbb X_n &=  I_g \otimes P_n\otimes Q_n, \\
  \mathbb X_n  \mathbb X_n^{\T} &=  I_{g^\prime} \otimes Q_n\otimes P_n.
 \end{split}
\end{equation}
\end{lemma}

\begin{rem}\rm
 \label{rem:analyticgg}
   In view of the definition of $\Bgg$, it is natural to think of
   an NC analytic function $h$ on $\Bgg$ as a function of the $g^\prime g$ variables $x_{j,\ell}$,
   $1\le j\le g^\prime$ and $1\le \ell \le g$.  In turn, 
   a monomial $m$ in $(x_{j,\ell})$ can be viewed as a {\it homogeneous}
   monomial  $u\otimes v$,
   where $u$ and $v$ are monomials of the same length (same as the length of $m$)
   and $u$ and $v$ monomials in NC variables $y_j$ ($1\le j\le g^\prime$)
    and $z_\ell$ ($1\le \ell \le g$) respectively.  In this way,
\begin{equation*}
    h= \sum_\alpha \sum_{|u|=|v|=\alpha} a_{u\otimes v} u\otimes v 
     = \sum_\alpha h^{(\alpha)}.
\end{equation*}
  For instance, the monomial $x_{23}x_{41}$ is identified with
  $y_2y_4 \otimes z_3z_1$.
\end{rem}

We want to evaluate NC analytic functions $\Bgg\to\OS{d}{d'}$ 
on $\mathbb X_n$,
which is a norm one matrix thereby causing power series convergence difficulties. However,
 evaluating NC analytic functions on \emph{nilpotent} tuples $X\in\Bgg$
 behaves especially well.
Here a tuple $X$ is called {\bf nilpotent}
 of order $\beta$ if $w(X)=0$ for every word $w$ of
 length $\geq \beta $.

\begin{lem}
 \label{lem:nil}
   If $f:\Bgg \to \OS{d'}d$ is NC analytic and $X\in\Bgg$ is
   nilpotent of order $\beta$, then $f(X)$ is defined and
   moreover,
 $$
  f(X)=\sum_{\alpha \le \beta} f^{(\alpha)}(X).
 $$ 
  In particular, if $f$ is an NC ball map, $f(0)=0$, and $Y\in\partial \Bgg$
  is nilpotent of order two, then
 $$
   f(Y)=f^{(1)}(Y).
 $$
\end{lem}
   
\begin{proof}
  Let $X\in\Bgg$ be given and let $r$ denote the series
  radius for $f$. For $z\in\mathbb D$ with $|z|<r$ the power series
  expansion for $f(zX)$ converges. The nilpotent hypothesis
  gives,
$$
 f(zX)= \sum_{\alpha \le \beta} f^{(\alpha)}(X)z^\alpha.
$$
 Since $f(zX)$ is analytic for $|z|<1$ and is equal to the polynomial
 on the right hand side above for $|z|<r$, equality holds for all $z$.

 If $Y\in \partial \Bgg$ and $Y$ is nilpotent of order two, the 
 argument above shows,
$$
 f(zY)= \sum_{\alpha \le 1} f^{(\alpha)}(Y)z^\alpha.
$$
 Moreover, the assumption $f(0)=0$ implies $f^{(0)}=0.$ Choosing $z=1$
 gives $f(Y)=f^{(1)}(Y)$.  
\end{proof}

\begin{lemma}
 \label{lem:uniqueS}
\begin{enumerate}[\rm (a)]
\item
   Suppose $p$ is an NC polynomial of degree $N$ with 
$\CC^{d'\times d}$ coefficients in $g'g$ variables and $p(0)=0$.
\begin{enumerate}[\rm (1)]
\item If 
 $$
   0 \preceq I-\mathbb X_n^{\T}\mathbb X_n - p(\mathbb X_n)^{\T} p(\mathbb X_n)
 $$ 
  for each $n\le N$,
  then $p=0$. 
\item
If 
 $$
   0 \preceq I-\mathbb X_n\mathbb X_n^{\T} - p(\mathbb X_n)p(\mathbb X_n)^{\T}
 $$
   for each $n\le N$, then $p=0$.
\end{enumerate}
\item
   Suppose $h:\Bgg\to\OS{d}{d'}$ is NC analytic.
If $h(\mathbb X_n)=0$ for each $n$, then $h=0$.
\end{enumerate}
\end{lemma}

\begin{proof}
(a) Write
$$
 p=\sum_{\alpha=0}^m p^{(\alpha)}
$$
 as in Remark \ref{rem:analyticgg}. In particular,
$$
  p^{(\alpha)} =\sum_{|u|=|v|=\alpha} a_{u\otimes v} u\otimes v,
$$
  and $a_{u\otimes v}\in \CC^{d^\prime\times d}$.

  By hypothesis $a_{\emptyset}=0$, so that $p^{(0)}=0$.
  Now  suppose $p_k=0$ for $k<n$.  Let
  $w$ be a word of length $n$ and $\gamma \in\CC^g$
  be given. From Lemma \ref{lem:funnyXn}, we have
$$
  0=(I-\mathbb X_n^{\T} \mathbb X_n)  \gamma \otimes w\otimes \emptyset
.
$$
  Hence, 
\begin{equation*}
   0 = p(\mathbb X_n) \gamma \otimes w\otimes \emptyset 
     =  p^{(n)}(\mathbb X_n) \gamma \otimes  w\otimes \emptyset 
     =  \sum_{|v|=n} p_{w\otimes v}\gamma \otimes \emptyset \otimes v.
\end{equation*}
Thus $p_{w\otimes v}=0$ and it follows that $p^{(n)}=0$.

(b) This proof 
  is similar. Here is a brief outline. 
  First note that $0=h(0)$. Let $r$ denote 
  the series radius for $h$. Fix $N$. For $|z|<r$ and for
  any $n\le N$, by Lemma \ref{lem:nil} we have (since $\mathbb X_n$
  is nilpotent of order $n\le N$)
$$
  h(\mathbb X_n) = \sum_{\alpha=1}^N h^{(\alpha)}(\mathbb X_n).
$$
  
  If we now let $p$ denote the polynomial $\sum_{\alpha=1}^Nh^{(\alpha)}$ of degree $N$,
it follows from (a) that 
   $p=0$. Since this is true for all $N$, we see $h=0.$
\end{proof}


\section{NC isometries}

This section has two parts. The first shows that the linear part
of an NC ball map is an NC ball map, i.e., it is what is
commonly known as a complete isometry.  The 
second subsection classifies these linear NC ball maps.
Recall that an NC analytic function $f:\Bgg \to \Bdd$
is an NC ball map provided it is NC analytic
and contraction-valued in the interior of $\Bgg$ and
for $X\in\Bgg(N)$ with $\|X\|=1$, $\|f(e^{it}X)\|=1$
for almost every $t\in\RR$. 

\subsection{Pencil ball maps have isometric derivatives}
 A linear mapping $\psi:\Mgg\to \Mdd$ is completely
 determined by its action on the matrix units $E_{j,\ell}\in \Mgg$
 with a $1$ in the $(j,\ell)$ position and $0$ elsewhere.
 The mapping $\psi$ then naturally extends to a mapping,
 still denoted $\psi$, on $\CC^{n\times n}\otimes \Mgg$ by the formula
\begin{equation}
 \label{eq:def1npsi}
 \psi\left(\begin{bmatrix} X_{j,\ell}\end{bmatrix}_{j,\ell}\right)
   =\sum X_{j,\ell}\otimes \psi(E_{j,\ell}) \in \CC^{n\times n} \otimes \Mdd.
\end{equation}
  For notational simplicity, the formula above is written $\psi(X)$.
  The mapping $\psi$ is {\bf completely isometric} if
  $\|\psi(X)\| = \|X\|$ for each $X\in \CC^{n	\times n}\otimes \Mgg$ and each $n$,
  and is {\bf completely contractive} if $\|\psi(X)\|\le \|X\|$
  for all $X$.

\begin{prop}\label{prop:linIsom}
 Suppose $f:\Bgg \to \Bdd$ is an NC analytic map with $f(0)=0$.
 If $f$ is an NC ball map, then
 $f^{(1)}$, the linear part of $f$, is a complete isometry.
\end{prop}

\begin{proof}
We start by observing that, in view of Lemma \ref{lem:nil},
\begin{equation}\label{eq:hish1}
 f^{(1)} \left(\left[\begin{array}{cc}0&X\\0&0\end{array}\right]\right)
= f \left(\left[\begin{array}{cc}0&X\\0&0\end{array}\right]\right)
\end{equation}
 for every $X\in\Bgg$.

If $f$ is an NC  ball map, then
for $X\in\partial\Bgg$ 
\begin{equation}\label{eq:hish12}
1=\|X\|=\left\| \left[\begin{array}{cc}0& X\\0&0\end{array}\right]\right\|=
\left\|f \left( \left[\begin{array}{cc}0& X\\0&0\end{array}\right] \right)\right\|
\end{equation}
by the binding property.
Now by \eqref{eq:hish1},
\begin{equation}\label{eq:hish13}
\left\|f \left( \left[\begin{array}{cc}0&X\\0&0\end{array}\right] \right)\right\|=
\left\|f^{(1)} \left( \left[\begin{array}{cc}0&X\\0&0\end{array}\right] \right)\right\|=
\left\| \left[\begin{array}{cc}0&f^{(1)}(X)\\0&0\end{array}\right] \right\|=
\|f^{(1)}(X)\|.
\end{equation}
From \eqref{eq:hish12} and \eqref{eq:hish13} we obtain 
$\|f^{(1)}(X)\|=1$ for all $X$ with $\|X\|=1$.
\end{proof}

\begin{rem}\rm
 \label{rem:shuffle}
  This remark does not contribute to the proofs, rather it is 
  for the sake of reconciling the definitions 
   of complete isometries and contractions given here with what
   is typically found in the literature (cf. \cite{Paulsen02}).

  Often a completely contractive (resp. isometric) mapping
  $\psi:\Mgg\to \Mdd$ is defined as follows. Given $n$, let
  $(\Mgg)^{n\times n}$ denote the $n\times n$ matrices with entries
  from $\Mgg$ and define
   $1_n\otimes \psi: (\Mgg)^{n\times n}\to (\Mdd)^{n\times n}$ by
$$
 1_n\otimes \psi\left(\begin{bmatrix} Y_{\alpha,\beta}\end{bmatrix}_{\alpha,\beta=1}^n\right)
  =\begin{bmatrix} \psi(Y_{\alpha,\beta})\end{bmatrix}_{\alpha,\beta=1}^n
.
$$
  
 In this definition, the block matrix $Y=\begin{bmatrix} Y_{\alpha,\beta}\end{bmatrix}_{\alpha,\beta=1}^n$ is written as
$$
  Y= \sum  E_{\alpha,\beta} \otimes Y_{\alpha,\beta},
$$
  where $E_{\alpha,\beta}\in \CC^{n	\times n}$ are the $n\times n$ matrix units.
  Evaluating $\psi$ on $Y$ becomes,
$$
 1_n\otimes \psi (Y) = \sum  E_{\alpha,\beta} \otimes \psi(Y_{\alpha,\beta}),
$$

  By using the matrix units basis $E_{j,\ell}$ of $\Mgg$,
$Y$ can be 
  rewritten as
$$
  Y=\sum X_{j,\ell}\otimes E_{j,\ell},
$$
  for some $X_{j,\ell}$.  Evaluating $1_n\otimes \psi$ on $Y$
  expressed as above gives equation \eqref{eq:def1npsi}.
  Passing between these two expressions for $Y$ is known
  as the {\bf canonical shuffle} in \cite{Paulsen02}.  

  Letting $A_{j,\ell}=\psi(E_{j,\ell})$, equation \eqref{eq:def1npsi} becomes,
$$
 \psi(X)=\sum X_{j,\ell}\otimes A_{j,\ell}.
$$
\end{rem}

\subsection{Completely isometric maps on $\Mgg$}
 \label{sec:bind}

 The following theorem which classifies completely isometric
maps on $\Mgg$ is the main result of this section.

 \begin{thm}
  \label{thm:complete-iso-structure}
   A linear mapping $\psi:\CC^{g^\prime\times g} \to \CC^{d^\prime\times d}$
   is completely isometric if and only if there exist
   unitaries $U:\CC^d \to \CC^d$, 
$V:\CC^{d^\prime}\to \CC^{d^\prime}$ and
 a completely contractive $($linear$)$
   mapping $\varphi:\CC^{g^\prime\times g} \to \CC^{(d^\prime - g^\prime)\times (d-g)}$
   such that 
  \begin{equation*}
    \psi(Y) = V \begin{bmatrix} Y & 0 \\ 0 & \varphi(Y) \end{bmatrix} U^{\T}.
  \end{equation*}
 \end{thm}

  Throughout this subsection let $\psi:\Mgg \to \Mdd$
  denote a completely isometric mapping.
  Let  $A_{j,\ell}=\psi(e_j (e_\ell^\prime)^{\T})$ 
  for $1\le j\le g'$ and $1\le \ell\leq g$ be as in Remark \ref{rem:shuffle}.
  We have represented $\psi$ in terms of the matrix
 \begin{equation*}
    A= \begin{bmatrix} A_{j,\ell} \end{bmatrix}_{j,\ell=1}^{g^\prime,g} \in
\left(\Mdd\right)^{g'\times g}
 \end{equation*}
  This matrix has the formal block transpose given by
 \begin{equation*}
   A^\T = \begin{bmatrix} A_{\ell,j} \end{bmatrix}_{j,\ell}.
 \end{equation*}

 \begin{lem}
  \label{lem:AT}
   If $\psi$ is completely contractive, then $A^\T$ is a contraction.
 \end{lem}

 \begin{proof}
   Choose 
  $X=\sum_{j,\ell=1}^{g,g^\prime} e_j (e_\ell^\prime)^{\T}\otimes (e_\ell^\prime)e_j^{\T}$.
   Direct computation reveals that $X^{\T}X = I_{gg^\prime}$ and thus
   the block matrix $X$ is a contraction. Hence
  \begin{equation*}
    \psi(X)=A^\T
  \end{equation*}
   is also a contraction.
 \end{proof}

 \begin{rem}\rm
\begin{enumerate}[\rm (1)]
\item
    That the converse of Lemma \ref{lem:AT} 
is not true in general can be seen by considering
    the mapping $\psi:\CC^{2\times 2}\to \CC$ defined by
    $\psi(e_j (e_\ell^\prime)^{\T})=\delta_j^\ell$.  In this case,
  \begin{equation*}
      A^\T= I_2,
  \end{equation*}
    but $\psi(E_{11} + E_{22})=2$, so that $\psi$ is not even contraction-valued.
\item
For $g=1$ the converse does hold. 
We leave this as an exercise for the interested reader.
\end{enumerate}
 \end{rem}

 \begin{prop}
  \label{prop:ball-map}
   A completely contractive mapping
   $\psi: \Mgg \to \Mdd$ is
   a complete isometry if and only if there exist  
   there is a set $\{f_1,\dots,f_g\}\subseteq \mathbb C^d$
   of unit vectors satisfying
 \begin{equation}
  \label{eq:fs}
   \langle A_{\alpha,s}f_u, A_{\beta,t}f_v\rangle =
     \begin{cases}   1 & \text{ if } (\alpha,s,t)=(\beta,u,v) \\
                     0 & \text{ otherwise.}
     \end{cases}
 \end{equation}
  Here $1\le u,v\le g$, $1\le s,t\le g,$ and $1\le \alpha,\beta \le g^\prime$.
 \end{prop}

  The following Lemma is an important ingredient in the proof.

\begin{lem}
 \label{lem:ortho-fs}
   Under the hypotheses of Proposition {\rm\ref{prop:ball-map}}, the set 
   $\{f_1,\dots,f_g\}$ is orthonormal.
   Moreover, 
 \begin{equation}
  \label{eq:hsa}
     h_\alpha=A_{\alpha,j}f_j\in\CC^{d'} \quad (1\leq\alpha\leq g')
 \end{equation}
    is independent of $j.$ 
\end{lem}

 \begin{proof}
      Let $f_j$ be a set of unit vectors satisfying equation \eqref{eq:fs}.
   Notice first that, for fixed $j$, the set 
   $\{A_{\alpha,j}f_j\mid 1\le \alpha \le g^\prime\}$
   is an orthonormal set. Let $\cS_j$ denote the span
   of this set. Given $j,\ell$ and $\alpha$,
$$
  A_{\alpha,j}f_j = \sum c_\beta A_{\beta,\ell}f_\ell +\zeta
$$ 
  for some $\zeta$ orthogonal
  to $\cS_\ell$ (and where the dependence of the coefficients 
  $c_\beta$ on $\alpha,j,\ell$
  has been suppressed).  Taking the inner product with
  $A_{\gamma,\ell}f_\ell$ it follows that $c_\beta=1$
  if $\beta=\alpha$ and $c_\beta =0$ otherwise; i.e.,
$$
  A_{\alpha,j}f_j = A_{\alpha,\ell}f_\ell  + \zeta.
$$
  On the other hand both $A_{\alpha,j}f_j$ and 
$A_{\alpha,\ell}f_\ell$ are unit vectors
  and thus $\zeta=0$.  Hence, $A_{\alpha,j}f_j$ is independent
  of $j$ and 
 \begin{equation*}
     h_\alpha = A_{\alpha,j}f_j
  \end{equation*}
   is unambiguously defined.

   Since $A_{\alpha,j}$ is a contraction (as it is, by definition,
   $\varphi(E_{\alpha,j})$) and since $\|f_j\|=1$ and
   $$\|A_{\alpha,j}f_j\|=1,$$ it follows that
  \begin{equation*}
    f_j = A_{j,\alpha}^{\T} h_\alpha,
  \end{equation*}
   and is thus independent of $\alpha$.

   Using this last claim, consider, for $j\ne \ell$,
  \begin{equation*}
     2\ge  \| (A_{j,\alpha}^{\T} + e^{it} A_{\ell,\alpha}^{\T})h_\alpha \|^2 
       =  2 + 2 \RE e^{it} \langle f_j,f_\ell \rangle.
  \end{equation*}
    It follows that $\langle f_j,f_\ell \rangle =0.$
    Here we have used
   $$
    \varphi\big((e_je_\alpha^{\T} \otimes e_\alpha e_j^{\T} 
     + e^{-it} e_\ell e_\alpha^{\T} \otimes e_\alpha e_\ell^{\T})\big)
     = A_{\alpha,j}+ e^{-it} A_{\alpha,\ell}
   $$
    is a contraction, $\left\|\begin{bmatrix}1& e^{it}\end{bmatrix}\right\|^2=2$, and  $\|h_\alpha\|=1.$
 \end{proof}

 \begin{proof}[Proof of Proposition {\rm \ref{prop:ball-map}}]
   Suppose such $f's$ exist. Let $X\in \CC^{g^\prime\times g}\otimes \CC^{n	\times n}$
   with $\|X\|=1$ be given. Thus $X$ is a contraction and there is a unit vector
   $x=\sum x_j \otimes e_j$ such that $\|Xx\|=1$.  In particular,
 \begin{equation*}
    \sum x_t^{\T} X_{\alpha,t}^{\T} X_{\alpha,s} x_s =1.
  \end{equation*}
   Thus,
 \begin{equation*}
  \begin{split}
   \langle  \psi(X)^{\T}\psi(X) \sum_u x_u \otimes f_u), \sum_v x_v\otimes f_v \rangle
      = & \sum (f_v^{\T} A_{\beta,t}^{\T} A_{\alpha,s}f_u)
             (x_v^{\T} X_{\beta,t}^{\T}X_{\alpha,s}x_u) \\
      = & \sum  x_t^{\T} X_{\alpha,t}^{\T} X_{\alpha,s}x_s = 1.
  \end{split}
 \end{equation*}

  Of course we also must be careful to check, in view of 
   the orthonormality of $\{f_1,\dots,f_g\}$ of Lemma \ref{lem:ortho-fs},
 \begin{equation*}
   \langle \sum_u x_u \otimes f_u, \sum_v x_v\otimes f_v \rangle
     = \sum_u  x_u^{\T} x_u =1.
 \end{equation*}
  Thus, if $\|X\|=1$, then $\|\psi(X)\|\ge 1$. Since $\psi$ assumed
  to be a contraction, the proof that $\psi$ is completely isometric  follows.

  Let us now turn to the converse. Suppose $\psi$ is completely isometric.
  Fix $\alpha$ and
  choose
  $X=\sum_\ell e_\alpha (e_\ell^\prime)^{\T} \otimes e_\alpha \otimes (e_\ell^\prime)^{\T}$. 
  Then,
 \begin{equation*}
   X X^{\T} = g^\prime e_\alpha e_\alpha ^{\T} \otimes e_\alpha  e_\alpha^{\T}.
 \end{equation*}
  Thus, $\varphi(X)=\sum A_{1,\ell}\otimes e_1 (e_\ell^\prime)^{\T}$ has norm at most
  $\sqrt{g^\prime}$. Equivalently,
 \begin{equation*}
    \Delta_\alpha = \begin{bmatrix} A_{\alpha,1} & \dots & A_{\alpha,g^\prime} \end{bmatrix}
 \end{equation*}
  has norm at most $\sqrt{g^\prime}$. Suppose now that
 \begin{equation*}
    h=\begin{bmatrix} h_1\\ \vdots \\ h_{g^\prime} \end{bmatrix}
 \end{equation*}
  and $\| \Delta_\alpha h \|^2 = g^\prime$. Then, using the fact that
  each $A_{\alpha,s}$ is a contraction,
 \begin{equation*}
   \begin{split}
     g^\prime &= \| \sum A_{\alpha,s}h_s \|^2 
        = | \sum_{s,t} \langle A_{\alpha,s}h_s, A_{\alpha,t}h_t \rangle | 
        \le   \sum_{s,t} | \langle A_{\alpha,s}h_s, A_{\alpha,t}h_t \rangle | \\
       & \le   \sum_{s,t} \| A_{\alpha,s}h_s\| \,  \| A_{\alpha,t}h_t \| 
       \le  \sum_{s,t} \|h_s \| \, \|h_t\| 
        =  (\sum \|h_s\|)^2 
       \le  g^\prime \|h\|^2 = g^\prime.
   \end{split}
 \end{equation*}
   The Cauchy-Schwartz inequality was used in two of the inequalities. Because
   equality prevails in the end, we must have
   equality in the inequalities. Therefore,
   $\|h_s\|^2 = \frac{1}{g^\prime}$ for each $s$ and moreover,
 \begin{equation*}
     \langle A_{\alpha,s}h_s, A_{\alpha,t}h_t \rangle =\frac{1}{g^\prime}
 \end{equation*}
   for each $s$.

    Choose $X=\sum_{j,\ell}  e_j(e_\ell^\prime)^{\T} \otimes e_j(e_\ell^\prime)^{\T}$ and
    note
   $\| X\|^2 = gg^\prime$. Then
\begin{equation*}
   \varphi(X)=   \sum  A_{j,\ell} e_j (e_\ell^\prime)^{\T}
\end{equation*}
   has norm squared exactly $gg^\prime$. In particular, there is a unit vector
\begin{equation*}
  f=\begin{bmatrix} f_1 \\ \vdots \\ f_{g^\prime} \end{bmatrix}
\end{equation*}
  such that $\|\varphi(X)f\|^2 = gg^\prime$. Hence
\begin{equation*}
  gg^\prime = \sum_\alpha \|\Delta_\alpha f \|^2.
\end{equation*}
 From the paragraph above $\|\Delta_\alpha\|^2 \le g^\prime$
 and thus for each $\alpha$ we must have $\|\Delta_\alpha f\|^2=g^\prime$.
 Again in view of the preceding paragraph, it follows that
 $\|f_s\|^2=\frac{1}{g^\prime}$ for each $s$ and moreover
\begin{equation}
 \label{eq:alpha-s-t}
  \langle A_{\alpha,s}f_s, A_{\alpha,t}f_t \rangle = \frac{1}{g^\prime}
\end{equation}
  for every $\alpha,s,t$.

  Fix $\alpha$.
  Applying the matrix $A^\T$ of Lemma \ref{lem:AT} to the vector
  $f_1\otimes e_\alpha$ produces the vector
 \begin{equation*}
    \begin{bmatrix} A_{\alpha,1}f_1 \\ A_{\alpha,2}f_1 \\ \vdots \\ A_{\alpha,g^\prime}f_1 \end{bmatrix}.
 \end{equation*}
  Since the first entry has norm $\sqrt{\frac{1}{g^\prime}}$
   and the whole vector itself has norm at
  most $\sqrt{\frac{1}{g^\prime}}$, it follows that
  $A_{\alpha,s}f_1=0$ whenever $s\ne 1$.  Applying the
  same argument to the other indices $u$ shows
 \beq
\label{eq:fisperp}
  A_{\alpha,s}f_u = 0 \quad \text{for } s\ne u.
\eeq

  For the final ingredient, fix $\alpha \ne \beta$ and let
 \begin{equation*}
   Y= e_\alpha (e_1^\prime)^{\T} \otimes e_1^{\T} +  e_\beta (e_2^\prime)^{\T} \otimes e_2^{\T}.
 \end{equation*}
  Since
 \begin{equation*}
    Y^{\T} Y = e_1e_1^{\T} \otimes e_1e_1^{\T}  + e_2e_2^{\T} \otimes e_2 e_2^{\T}
,
 \end{equation*}
   $Y$ is a contraction. Therefore,
 \begin{equation*}
   \varphi(Y)= A_{\alpha,1} \otimes e_1^{\T} + A_{\beta,2}\otimes e_2^{\T}
 \end{equation*}
  is also a contraction.  Let
 \begin{equation*}
   F(t) = f_1\otimes e_1 + e^{it} f_2 \otimes e_2.
 \end{equation*}

  With these notations,
\begin{equation*}
  \varphi(Y)F(t) = A_{\alpha,1}f_1 + A_{\beta,2}f_2,
\end{equation*}
  which gives the second equality in 
 \begin{equation*}
\begin{split}
    2 &= \frac{1}{2\pi}\int\big( 2+e^{-it}\langle A_{\alpha,1}f_1,A_{\beta,2}f_2 \rangle
                             + e^{it} \langle A_{\beta,2}f_2,A_{\alpha,1}f_1 \rangle \big)dt\\
      &= \frac{1}{2\pi} \int \| \varphi(Y)F(t)\|^2 dt 
      \le  2.
  \end{split}
\end{equation*}
  The inequality is a consequence of the hypothesis $\|\varphi(Y)\|\le 1$
  and $\|F(t)\|^2=2$.  It follows that $\|\varphi(Y)F(t)\|=1$ for
  every $t$ and thus 
   $\langle A_{\alpha,1}f_1, A_{\beta,2}f_2\rangle =0$ whenever $\alpha\ne \beta$.

   Repeating the argument with other indices shows,
 \begin{equation}
  \label{eq:alpha-not-beta}
    \langle A_{\alpha,s}f_s, A_{\beta,t}f_t\rangle =0 \mbox{ if } \alpha\ne \beta.
 \end{equation}
   (Here $s=t$ is ok so long
   as $\alpha \ne \beta$.)

   Combining equations \eqref{eq:alpha-s-t},
   \eqref{eq:fisperp}, and \eqref{eq:alpha-not-beta} gives
the desired \eqref{eq:fs}.
\end{proof}

\subsubsection{Characterization of complete isometries}
  In this subsection, Theorem \ref{thm:complete-iso-structure} is deduced from Proposition \ref{prop:ball-map}. We begin with a lemma which follows readily from Lemma \ref{lem:uniqueS}.

\begin{lem}
 \label{lem:uniqueS-special}
   Suppose the linear map $\Sigma:\Mgg\to \Mdd$ has the form
$$
  \Sigma(x)=\begin{bmatrix} x & \sigma_1(x)\\ \sigma_2(x) & \sigma_3(x)\end{bmatrix}.
$$
  If $\Sigma$ is a completely contractive, then $\sigma_1=0$ and
$\sigma_2=0$.
\end{lem}

\begin{proof}
  For a given $n$ we have
\begin{equation*}
 \begin{split}
  0&\preceq  I-\Sigma(\mathbb X_n)^{\T} \Sigma(\mathbb X_n) \\
   &=  \begin{bmatrix} I-\mathbb X_n^{\T} \mathbb X_n 
         - \sigma_2(\mathbb X_n)^{\T}\sigma_2(\mathbb X_n) & * \\ * & * \end{bmatrix}
 \end{split}
\end{equation*}
  Thus the upper left hand corner in the block matrix above is positive
  semidefinite and Lemma \ref{lem:uniqueS} implies $\sigma_2=0$.
  Reversing the order of the products shows $\sigma_1=0$.
\end{proof}

 \begin{proof}[Proof of Theorem {\rm \ref{thm:complete-iso-structure}}]
  If $\psi$ has the given form, then $\psi$ is evidently
  completely isometric.  

  Conversely, suppose $\psi$ is completely isometric.
  Let $f_j$ be a set of unit vectors satisfying equation \eqref{eq:fs}.
  By Lemma \ref{lem:ortho-fs}, the set $\{f_1,\dots,f_g\}$
  is orthonormal and moreover, $h_\alpha=A_{\alpha,j}f_j$
  is independent of $j$ and $\{h_1,\dots,h_{g^\prime}\}$ is
  also an orthonormal set.

   Let
 \begin{equation*}
     F=   \begin{bmatrix} f_1 & \dots & f_{g} \end{bmatrix}, \quad
     H=  \begin{bmatrix} h_1 & \dots & h_{g'} \end{bmatrix}.
 \end{equation*}
  The mappings $F,H$ are isometries $\CC^g \to \CC^d$ 
and $\CC^{g^\prime} \to \CC^{d^\prime}$
 respectively.  Further,
  for given $\beta,u$, 
 \begin{equation*}
  \begin{split}
    h_\beta^\T  \sum_{\alpha,s} x_{\alpha,s} A_{\alpha,s} f_u
      =  \sum x_{\alpha,s} h_\beta^\T  A_{\alpha,s}f_u 
      = x_{\alpha,s} h_\alpha^\T  A_{\alpha,s}f_s 
      = x_{\alpha,s}.
  \end{split}
 \end{equation*}
   It follows that, 
 \begin{equation*}
   H^{\T} \varphi(x) F = x.
 \end{equation*}
   This proves the first part of this direction of the theorem. 

   The isometries $H$ and $F$ extend to unitaries $V$ and $U$ respectively
   which produces the representation
 $$
   \varphi(x)= V \begin{bmatrix} x & \sigma_1 \\ \sigma_2 & \sigma_3 \end{bmatrix} U^{\T},
 $$
   where the block matrix $\Sigma=\begin{bmatrix} x & \sigma_1 \\ \sigma_2 & \sigma_3 \end{bmatrix}$ is completely contractive since 
   the same is true of $\varphi$. Now Lemma \ref{lem:uniqueS-special}
   completes the proof. 
\end{proof}

\section{Proof of Theorem \ref{thm:ball-NO-L}}
\label{sec:scottsPf}
  In this section we prove Theorem \ref{thm:ball-NO-L}. Accordingly,
  suppose $h:\Bgg\to \Bdd$ is an NC  ball map
  and  $h(0)=0.$ 
  From Lemma \ref{prop:linIsom},
  $h^{(1)}$, the linear part of $h$, is a complete isometry.
  By Theorem \ref{thm:complete-iso-structure},  
  there exists unitaries $\U$ and
  $V$ and a completely contractive mapping $\tilde{\hh}^{(1)}$ such that
 \begin{equation}
  \label{eq:lin}
   \hh^{(1)}(x)= V \begin{bmatrix} x & 0 \\ 0 & \tilde{\hh}^{(1)}(x) \end{bmatrix} \U^{\T}.
 \end{equation}
  We claim that $Vh(x)U^\T$ is of the desired form \eqref{eq:ballDeal}.

 For the sake of convenience we replace $h(x)$ by $V^{\T}h(x)U$.
For $X\in\Bgg(N)$ consider
$\DD\to\OS{d}{d'}(N)$, $z\mapsto h(zX)$, which is analytic (in $z$).
This is a function of one complex variable, so the classical
Schwarz lemma applies. Hence
for all $0\leq\delta<1$
and $0\leq \theta \leq 2\pi$ we have
\begin{equation}\label{eq:hIneq01}
0\preceq \delta^2 I -h(\delta e^{i\theta}X)^\T  h(\delta e^{i\theta}X).
\end{equation}
If $\delta$ is in
the series radius, we may write $$h(\delta e^{i\theta}X)=
h^{(1)}(\delta e^{i\theta}X)+ h^{(\infty)}(\delta e^{i\theta}X) =
\sum_{\alpha=1}^\infty h^{(\alpha)}(\delta e^{i\theta}X).$$
We integrate \eqref{eq:hIneq01} for such $\delta$ to obtain
\begin{equation}\begin{split}\label{eq:hIneq00} 0 &\preceq
\frac 1{2\pi}
\int^{2\pi}_0\big(\delta^2 I -h(\delta e^{i\theta}X)^\T  h(\delta e^{i\theta}X)\big)
d\theta
\\ &=
\delta^2 I-\delta^2h^{(1)}(X)^\T h^{(1)}(X) - \frac 1{2\pi}
\int^{2\pi}_0 h^{(\infty)}(\delta e^{i\theta} X)^\T  h^{(\infty)}(\delta e^{i\theta}X) d\theta \\
&=
\delta^2 I -\delta^2h^{(1)}(X)^\T h^{(1)}(X) 
   - \sum_{\alpha=2}^\infty \delta^{2\alpha}h^{(\alpha)}(X)^\T  h^{(\alpha)}(X),
\end{split}
\end{equation}
  where the last equality uses the homogeneity (of order $\alpha$) of
   $h^{(\alpha)}$.

Fix an $\alpha\geq 2$ and
write $\delta^{\alpha-1} h^{(\alpha)}=\left[\begin{array}{cc}
b_1&b_2\\
b_3&b_4
\end{array}\right]$ for NC analytic polynomials $b_j$. Then
by equations \eqref{eq:lin} and \eqref{eq:hIneq00} and because
the $b_j$ are polynomials,
\beq\label{eq:crucial1}
\begin{split}
0 & \preceq \left[\begin{array}{cc}
I&0 \\
0 & I
\end{array}\right]
-
\left[\begin{array}{cc}
X^\T X &0 \\
0 & \tilde h^{(1)}(X)^\T \tilde h^{(1)}(X)
\end{array}\right]
-
\left[\begin{array}{cc}
b_1(X)^\T&b_3(X)^\T\\
b_2(X)^\T&b_4(X)^\T
\end{array}\right] \left[\begin{array}{cc}
b_1(X)&b_2(X)\\
b_3(X)&b_4(X)
\end{array}\right] \\
&=
\left[\begin{array}{cc}
I-X^{\T} X &0 \\
0 &I - \tilde h^{(1)}(X)^\T \tilde h^{(1)}(X)
\end{array}\right]
 \\&\qquad - \left[\begin{array}{cc}
b_1(X)^\T b_1(X)+b_3(X)^\T b_3(X) & b_1(X)^\T b_2(X)+ b_3(X)^\T b_4(X) \\
b_2(X)^\T b_1(X)+ b_4(X)^\T b_3(X) & b_2(X)^\T b_2(X)+b_4(X)^\T b_4(X)
\end{array}\right].
\end{split}
\eeq

 It follows that
$$
  I-\mathbb X_n^{\T} \mathbb X_n - b_j(\mathbb X_n)^{\T} b_j(\mathbb X_n) \succeq 0
$$ 
  for $j=1,3$ and all $n$.  Lemma \ref{lem:uniqueS} thus implies
  $b_1=0$ and $b_3=0$.

  We now multiply in the other order (consider say $\mathbb X_n \mathbb  X_n^{\T}$
  instead of $\mathbb X_n^{\T} \mathbb X_n$) to conclude that $b_2=0$
  (also $b_1=0,$ but that we already knew.) This shows $h$ has the desired
  form and completes the proof.


\section{Linear fractional transformation of a ball} 
 \label{sec:lfl}
It is well known that the bianalytic maps 
on the unit disk $\DD$ are exactly the linear fractional maps.
These act transitively on the unit disk. That is, if $w, z \in \DD$,
then there is a linear fractional map $\cF$ which maps $w$ to $z$. 
It is standard in classical
several complex variables that this generalizes to special domains in 
$\CC^n$ \cite{He}. In this subsection
we give basic properties of linear fractional maps on $\Bdd$.

Given a $d^\prime \times d$ matrix $v$ with $\|v\|<1$, define
$\cF_v:\Bdd \to \Bdd$ by 
\beq\label{eq:deff}
\cF_v (u):=v-(I_{d'}-vv^\T )^{1/2}u(I_d-v^\T u)^{-1}(I_d-v^\T v)^{1/2}.
\eeq

\begin{lem}\label{lem:linFrac1}
Suppose $\cD$ is an open NC domain containing $0$. If  $u:\cD \to \Bdd$
is NC analytic, then 
$\cF_v (u(x))$ is an NC analytic function $($in $x)$ on $\cD$.
\end{lem}

\begin{proof}
Since $v$ is a matrix with $\|v\|<1$, the expressions $(I_{d'}-vv^\T )^{1/2}$ and
$(I_d-v^\T v)^{1/2}$ are constant NC analytic
functions. As sums and products of NC analytic functions are
NC analytic, it suffices
to show that $\zeta:=(I_d-v^\T u)^{-1}$ is NC analytic. Note that $v^\T u$ is
NC analytic on $\cD$. Thus $\zeta$ being the composition of
the NC analytic function $(1-z)^{-1}$ on $\DD$ and the NC analytic
function $v^\T u$ on $\cD$ is NC analytic as well.
\end{proof}

Now we give the basic properties of $\cF$ in a lemma
generalizing Lemma \ref{lemma:multilinfracIntro}.
For this we
define $\cU_k$ to be the set of all $ U \in \Bdd(N)$
which are isometric on a space of dimension at least
$ N k$. For example, 
$\cU_d$ denotes the isometries in $\Bdd(N)$.

\begin{lemma} \label{lemma:multilinfrac}
Suppose that  $N\in \mathbb{N}$ and $V\in\Bdd(N)$ with $\| V\|
<1$.
\begin{enumerate}[\rm (1)]
\item $U\mapsto \cF_V(U)$ maps the unit ball $\Bdd(N)$ into itself
with boundary to the boundary.
Furthermore, for each $k\leq d$, $\cU_k$ maps onto $\cU_k$.
\item If $U\in\Bdd(N)$, then
$\cF_V(\cF_V(U))=U.$
\item $\cF_V(V)=0$ and $\cF_V(0)=V$.
\end{enumerate}
\end{lemma}

\begin{proof}
The proof is motivated by linear system theory but an understanding of system theory is not
needed to read the proof.

  Let $y\in\mathbb C^{Nd}$ be given.  Define 
\begin{equation*}
 i= \begin{pmatrix} i_1\\i_2 \end{pmatrix}
    =\begin{pmatrix} (I-V^\T V)^{-\frac12} (I-V^\T U)y \\ -Uy \end{pmatrix}
     \in  \mathbb C^{Nd}  \oplus  \CC^{Nd^\prime}.
\end{equation*}

 Let $M$ denote the matrix 
\beq
 M:=\bmat (I-V^\T V)^{1/2} & -V^\T  \\ V & (I-VV^\T )^{1/2} \emat.
\eeq
 Straightforward computation shows $M$ is unitary; i.e., $M^\T M=I=MM^\T$.
 Let
\begin{equation*}
 \begin{split}
  o= \begin{pmatrix} o_1 \\ o_2\end{pmatrix} 
   =  Mi  =  \begin{pmatrix} y \\ (I-VV^\T)^{-\frac12} (V-U)y \end{pmatrix}
    \in  \mathbb C^{Nd}  \oplus  \CC^{Nd^\prime}.
 \end{split}
\end{equation*}
   The relation $V(I-VV^\T)^{-\frac12}= (I-V^\T V)^{-\frac12}V$ 
   was used in computing $Mi$. 

   Since $M$ is unitary,
\begin{equation}
 \label{eq:multiio} 
   \|i_1\|^2+  \|i_{2}\|^2=\|o_1\|^2+\|o_{2}\|^2.
\end{equation}
  On the other hand, computations give
\begin{equation*}
  \cF_V(U)i_1=o_2.
\end{equation*}
 Combining the last two equations gives
\begin{equation}
 \label{eq:FUVy}
 \begin{split}
  \|i_1\|^2-  \|\cF_V(U)i_1\|^2 = & \|o_1\|^2-\|i_2\|^2 
           =  \|y^2\|^2-\|Uy\|^2
           \ge 0.
 \end{split}
\end{equation}
  Since the mapping $y\mapsto i_1=(I-V^\T V)^{-\frac12} (I-V^\T U)y$
  is onto, the matrix $\cF_V(U)$ is a contraction
  and the first part of item (1) of  the lemma is proved.

  To prove the second part of item (1), notice that from
  equation \eqref{eq:FUVy} and the fact that
  both $\cF_V(U)$ and $U$ are contractions, the dimension of the space
  on which $\cF_V(U)$ is isometric is the same
  as the dimension of the space on which $U$ is isometric.

  We now turn to the proof of item (2).
Define $$F := \cF_V(U) = V - (I - VV^\T )^{1/2}U(I-V^\T U)^{-1}(I- V^\T V)^{1/2}.$$
First notice that
\begin{equation*}\begin{split}I-V^\T F&=I -
V^\T V + V^\T (I - VV^\T )^{1/2}U(I-V^\T U)^{-1}(I- V^\T V)^{1/2} \\
&= (I -
V^\T V) + (I - V^\T V)^{1/2}V^\T U(I-V^\T U)^{-1}(I- V^\T V)^{1/2} \\
&= (I - V^\T V)^{1/2}(I-V^\T U)(I-V^\T U)^{-1}(I- V^\T V)^{1/2} +  \\
& \qquad + (I - V^\T V)^{1/2}V^\T U(I-V^\T U)^{-1}(I- V^\T V)^{1/2} \\
&= (I - V^\T V)^{1/2}(I-V^\T U)^{-1}(I- V^\T V)^{1/2}.
\end{split} \end{equation*}
So
$$(I-V^\T F)^{-1}=(1-V^\T V)^{-1/2}(I-V^\T U)(I-V^\T V)^{-1/2}.$$
We use this and elementary calculations to obtain
\begin{equation*}
\begin{split}
\cF_V(F)&=V-(I-VV^\T )^{1/2}F(I-V^\T F)^{-1}(I-V^\T V)^{1/2} \\
&=V-(I-VV^\T )^{1/2}F(I-V^\T V)^{-1/2}(I-V^\T U) \\
&=V-(I-VV^\T )^{1/2}V(I-V^\T V)^{-1/2}(I-V^\T U)+(I-VV^\T )U \\
&=V-V(I-V^\T U)+U-VV^\T U=U.
\end{split}
\end{equation*}
For (3), compute
\begin{equation*}
\begin{split}
\cF_V(V)&= V- (I-VV^\T )^{1/2}V(I-V^\T V)^{-1}(I-V^\T V)^{1/2} \\
&= V-(I-VV^\T )^{1/2}V(I-V^\T V)^{-1/2} \\
&= V-V(I-V^\T V)^{1/2}(I-V^\T V)^{-1/2}=0.
\end{split}
\end{equation*}
\end{proof}

\part*{\centerline{Part II. Clinging}}
 \label{partII}
  In this, and the sections to follow, we turn  our attention
  to semi-distinguished ball maps introduced in 
  \S \ref{subsec:distinguished}.  In particular,
  attention is restricted to the NC
  domains $\cB_{g^\prime}$.

\section{NC functions revisited}
\label{sec:NCanalAgain}

This section gives several basic facts about NC analytic functions
on the ball, most of which are used in the remainder of the paper.
We feel several of the main results here also are of interest in
their own right. A few of the results are included purely for their
own sake.

\subsection{Series radius of convergence}

This section shows that NC power series expansions
of NC analytic functions on a ball have good convergence
properties. As a consequence of this convergence, 
  bounded NC analytic functions are free
  analytic in the sense of Popescu \cite{Pop1}.

 \begin{lemma}
  \label{lem:abs-converge-contractive}
    If $h:\cB_g\to \Bdd$ is an NC analytic function, 
   with NC power series expansion
$$
  h=\sum_w a_w w,
$$
   then 
  \begin{equation*}
    \sum_w \|a_w\|^2 \le d.
  \end{equation*}
   Moreover, if $Z$ is a strict column contraction acting
   on a separable Hilbert space or if
   $Z=I\otimes S^\T $ where $S$ is the shift of Fock space $\mathcal F_g$, 
   and $z\in\DD$, then
  \begin{equation*}
    h(zZ) = \sum a_w \otimes (zZ)^w
  \end{equation*}
   converges absolutely, $h(zZ)$ is a contraction and 
$z\mapsto h(zZ)$ is an analytic function on $\DD$.
 \end{lemma}

 \begin{proof}
    Let $S$ denote the shifts introduced in \S \ref{sec:model}. Let $\mathcal F_g(n)$
    denote the span of the words of length at most $n$
    in the NC Fock space $\mathcal F_g$.
    Let $W_n:\mathcal F_g(n)\to \mathcal F_g$ denote the inclusion.
    Thus, for any finite dimensional Hilbert space $\mathcal K$, 
     $I\otimes S_j(n)=I\otimes W_n^\T (I\otimes S_j) I\otimes W_n$ is
    the compression
    of  $I\otimes S_j$ to the (semi-invariant finite dimensional) 
    subspace $\mathcal K \otimes \mathcal F_g(n)$. 
    Here $I$ is the identity on $\mathcal K$.

    In view of the hypotheses (and since the $S_j(n)$ are nilpotent
     of order $n$),
   $$ 
      h(S(n)^\T)=\sum_{|w|\le n} a_w \otimes w(S(n))^\T.
  $$
    Thus, for any vector $\gamma \in \CC^d$,
  \begin{equation*}
   \begin{split}
    \|\gamma\|^2  \ge   \|h(S(n)^\T)^\T \gamma \otimes \emptyset \|^2 
       = \big\| \sum_{|w|\le n} a_w^\T \gamma \otimes w \big\|^2 
       = \sum_{|w|\le n} \|a_w^\T \gamma \|^2.
   \end{split}
  \end{equation*}
    It follows that,
  \begin{equation*}
     d\ge \sum_w \sum_j \| a_w^\T e_j\|^2,
  \end{equation*}
    where $\{e_1,\dots,e_d\}$ is an orthonormal basis for $\CC^d$.
    (Note that the sums over $j$ terms on the right hand side are the squares
     of the Hilbert-Schmidt norms of the $a_w$). Since
   $\|a_w\|^2=\|a_w^\T\|^2 \le \sum_j \|a_w^\T e_j\|^2$, it follows that
  \begin{equation*}
   d\ge \sum \|a_w\|^2.
  \end{equation*}

   Consequently, if $|z|<1$ and $Z=(Z_1,\dots,Z_g)$
   is a $g$ tuple of operators on Hilbert space
   (potentially infinite dimensional) satisfying
   $\sum Z_j^\T  Z_j \le  I$ and 
   if $|z|<1$, then
 \begin{equation*}
   h(zZ):=\sum a_w \otimes (zZ)^w
 \end{equation*}
   converges (absolutely). A favorite choice is $Z=I\otimes S^{\T}.$

   For $|z|<1$, we have $I\otimes W_n W_n^{\T} h(zI\otimes S^{\T})I\otimes W_nW_n^{\T}$
   converges in the SOT to $h(zI\otimes S^{\T})$. On the other hand,
   $I\otimes W_n^{\T} h(zI\otimes S^{\T})I\otimes W_n= h(z I\otimes S(n)^{\T})$ which is assumed to
   be a contraction. Thus, $h(zI\otimes S^{\T})$ is a contraction.

   For a general strict column contraction $X$, represent
   $X$ as $VX=(I\otimes S^{\T}) V$ by Lemma \ref{lem:model}. For $|z|<1$, it follows that
    $h(I\otimes zS^{\T})V= Vh(zX)$ and hence $\|h(zX)\|\le 1$.
 \end{proof}

\subsection{The NC Schwarz lemma}
\label{subsec:schwarz}
The classical Schwarz lemma from complex variables states the
following:
if $f:\mathbb D\to\mathbb D$ is analytic and $f(0)=0$, then
$\|f(z)\|\leq \|z\|$ for $z\in\DD$.
There are several ways to extend this to NC analytic functions, for 
example Popescu \cite[Theorem 2.4]{Pop1} gives one. 
In this subsection we give two extensions of our own.

 \begin{thm}
  \label{thm:schwarz-lemma}
     Suppose $f:\cB_{g^\prime} \to \cB_{d^\prime\times d}$
    is an NC analytic function on $\cB_{g'}.$
    If 
    $f(0)=0$ and $\|f(X)\|\le 1$ for each $X\in\Int\cB_{g'}$,
    then
   \begin{equation}\label{eq:schwarzBinds}
     X^{\T}X - f(X)^{\T} f(X) \succeq 0,
  \end{equation}
    for $X\in\Int\cB_{g'}$.
 \end{thm}

 \begin{proof}
    The proof relies on the model for
    column contractions and 
    the convergence result for
    bounded NC analytic functions $f$ of Lemma \ref{lem:abs-converge-contractive} 
    which allows us to evaluate bounded NC analytic functions on 
    operators, not just matrices.
    
Since $f$ maps into $\cB_{d'\times d}$, if $|z|<1$, then
\beq\label{eq:schwarz1}
I-f(zS^\T )^\T f(zS^\T )\succeq 0,
\eeq
by Lemma \ref{lem:abs-converge-contractive}. 
 Thus,
\beq\label{eq:schwarz2}
\sum_j S_jS_j^\T +P_0-f(zS^\T )^\T f(zS^\T )\succeq0.
\eeq
(Here $P_0$ is the projection onto the span of the empty word.)
From $f(0)=0$, we obtain $f(S^\T )P_0=0$. Hence \eqref{eq:schwarz2} transforms
into
\beq\label{eq:schwarz3}
\begin{split}
&(I-P_0) \Big( \sum_j S_jS_j^\T +P_0-f(zS^\T )^\T f(zS^\T ) \Big) (I-P_0)
+ P_0 \Big( \sum_j S_jS_j^\T +P_0-f(zS^\T )^\T f(zS^\T ) \Big) P_0= \\
&
(I-P_0) \Big( \sum_j S_jS_j^\T +P_0-f(zS^\T )^\T f(zS^\T ) \Big) (I-P_0)
+P_0 \succeq 0.
\end{split}
\eeq
As 
$$
(I-P_0) \Big( \sum_j S_jS_j^\T +P_0-f(zS^\T )^\T f(zS^\T ) \Big) (I-P_0)
=(I-P_0) \Big( \sum_j S_jS_j^\T -f(zS^\T )^\T f(zS^\T ) \Big) (I-P_0)
$$
and 
$$
P_0 \Big( \sum_j S_jS_j^\T -f(zS^\T )^\T f(zS^\T ) \Big)=0=
 \Big( \sum_j S_jS_j^\T -f(zS^\T )^\T f(zS^\T ) \Big) P_0,
$$
\eqref{eq:schwarz3} is equivalent to
\beq
\sum_j S_jS_j^\T -f(zS^\T )^\T f(zS^\T ) \succeq 0,
\eeq
    for $|z|<1$.
  Replacing $S$ by $I\otimes S$ in the argument above yields,
\begin{equation}
 \label{eq:schwarz4}
    I\otimes \sum_j S_jS_j^\T -f(zI\otimes S^\T )^\T f(zI\otimes S^\T ) \succeq 0.
\end{equation}

    Given
     $X\in\Bg$ with $\|X\|<1$, we can write
   $VX=(I\otimes S^{\T})V$, where $I$ is the identity
  on a finite dimensional Hilbert space,
  by Lemma \ref{lem:model}. Moreover, by
  Lemma \ref{lem:abs-converge-contractive}, for $|z|<1$, 
$$
  V f(zX)= f(zI\otimes S^\T )V.
$$
  Multiply \eqref{eq:schwarz4} by $V^\T$ on the left and $V$ on the right to obtain
$$
 V^\T \Big(\sum_j I\otimes S_jS_j^\T -f(zI\otimes S^\T )^\T f(zI\otimes S^\T ) \Big) V=
  X^{\T}X - f(zX)^{\T} f(zX) \succeq 0,
$$
  for $|z|<1$. 
 Since $\|X\|<1$, letting $z\nearrow 1$ completes  the proof.
\end{proof}

 \begin{rem}
   Popescu \cite[Theorem 2.4]{Pop1} formulates and proves a Schwarz lemma for free analytic
   functions, which in our context implies that if $f$ is a contraction-valued  NC
   analytic
   function with $f(0)=0$, then $\|f(X)\|\le \|X\|$ for $\|X\|<1$ and further,
   $\sum_{|w|=\alpha} a_{w}a_{w}^{\T} \leq I$ for all $\alpha$.
   (This inequality remains true
   even with operator coefficients $a_w$.)
 \end{rem}

A classical complex variables statement equivalent to Schwarz's lemma is the following:
 if $f:\mathbb D\to\mathbb D$ is analytic and $f(0)=0$, then
 $h(z)=\frac{f(z)}{z}$ is also analytic and $h:\mathbb D\to \mathbb D$.
We give a noncommutative analog of this result, which 
does not appear to be an immediate
consequence of Theorem \ref{thm:schwarz-lemma}.

\begin{thm}
  \label{thm:scott}
    Suppose that $H=\left[\begin{array}{cccc}
    H_1 & \dots & H_{g'}\end{array}\right]$ is a row
    of $d'\times d$ NC analytic functions on $\cB_{g'}$.
    If for each $X\in\Int\cB_{g'}$,
  \begin{equation}\label{eq:ass1}
    \|H(X)\VX \|= \| \sum_j H_j(X)X_j\|\le 1,
  \end{equation}
i.e.,
\beq\label{eq:ass2}
I-H(X)\VX\ (H(X)\VX)^\T \succeq 0,
\eeq
   then for each $X\in\Int\cB_{g'}$
  \begin{equation}\label{eq:conc1}
    I -     H(X) H(X)^\T \succeq 0.
  \end{equation}
   Equivalently, $\|H(X)\|\le 1$.
\end{thm}

\begin{proof}
This proof depends upon both Lemmas \ref{lem:model}
and \ref{lem:abs-converge-contractive}. 

 Let $G(x)=H(x)x$. The hypotheses imply $G:\cB_g \to \OS{d'}d$
 is contraction-valued.  Hence Lemma \ref{lem:abs-converge-contractive}
 applies.  
 Denote the power series expansions for $H_j$ by
$$
  H_j=\sum_\alpha h_j^{(\alpha)}.
$$
 It follows that the power series expansion (by homogeneous terms) for $G$ is then
$$
  G=\sum_\alpha \sum_j  h_j^{(\alpha)}x_j.
$$
 Hence, also by Lemma \ref{lem:abs-converge-contractive}, for each $j$
 the power series expansion for $H_j$ converges for any
 strict column contraction $Z$ (even for operators on an infinite dimensional
 Hilbert space) and for such $Z$,
$$
 G(Z)=\sum_j H_j(Z) Z_j.
$$
  In particular, for $|z|<1$ and $Z=zI\otimes S^\T $, 
 (where $S$ is an in Lemma \ref{lem:model} and $I$ 
  is the identity on a finite dimensional Hilbert space),
$$
 G(zS^\T ) = \sum_j H_j(zI\otimes S^\T ) S_j^\T .
$$
 
 Because $\|G(zI\otimes S^\T )\|\le 1$, 
\begin{equation}
 \label{eq:schwarz2a}
 \begin{split}
  0& \preceq I-G(zI\otimes S^\T ) G(zI\otimes S)^\T  \\
  &= I -\sum_j H_j(zI\otimes S^\T ) I\otimes S_j^\T  \sum_\ell I\otimes S_\ell H_\ell(zI\otimes S^\T )^\T  \\
  &= I - \sum_j H_j(zI\otimes S^\T ) H_j(zI\otimes S^\T )^\T .
 \end{split}
\end{equation}

   Let $X\in\Int\cB_{g'}$ be a strict column contraction acting on
   a finite dimensional space. Express $X=V^\T( I\otimes S^\T )V$ according to
   Lemma \ref{lem:model},
  where $I$ is the identity on a finite dimensional Hilbert space.
  For every NC analytic polynomial $f$ and $|z|<1$,
  $f(zX)=V^\T f(zI\otimes S^\T) V$. Hence the same
  holds true for NC analytic functions and in particular,
$$
  H_j(zI\otimes S^\T ) V= VH_j(zX).
$$
  Thus, applying $V$ on the right and $V^\T $ on the left of 
  equation \eqref{eq:schwarz2a} gives,
$$
 0\preceq I - \sum_j H_j(zX) H_j(zX)^\T .
$$
Letting $z\nearrow 1$ concludes the proof.
\end{proof}

\subsection{The distinguished boundary for $\Bgg$}
  \label{subsec:distinguished-boundary}
Fix $N.$
The distinguished (Shilov) boundary of the algebra
$\mathcal A(\Bgg(N)),$ the
functions which are analytic in $\Int\Bgg(N)$ and continuous on $\Bgg(N)$
is the smallest closed subset
$\Delta$ of $\Bgg(N)$ so that each element of $\mathcal A(\Bgg(N))$
takes is maximum on $\Delta$. That a smallest, as opposed
simply minimal, such sets exists is a standard fact in
complex analysis and the theory of uniform algebras; see
\cite[p. 145]{Kr} or \cite[Ch. 4]{He} for more details.

  While not needed in the sequel, the following 
  known result explains the 
 {\it distinguished} terminology in the definitions
  of distinguished isometry and semi-distinguished 
  pencil ball map. 

\begin{prop}
 \label{prop:BoundaryBgg}
 The distinguished boundary of $\mathcal A(\Bgg(N))$
 is $\{X\in\Bgg(N) \mid X^\T X=I\}$.
\end{prop}

  That the distinguished boundary
  of $\mathcal A(\Bgg(N))$ must be contained in
  $\{X\in\Bgg(N) \mid X^\T X=I\}$ follows 
  readily from \ref{lemma:multilinfrac};
  see Proposition \ref{prop:boundaryBgg}.
  For the fact that no smaller set can
  serve as a distinguished boundary, we
  refer to reader to \cite[p. 77]{Ab}.
  
\begin{prop}
 \label{prop:boundaryBgg}
   Fix $N\in\NN$. 
  If $f: \Bgg(N) \to \Mdd$ is continuous 
  and analytic in $\Int\Bgg(N)$,
 then for any  $X \in \Bgg(N) $ we have
\beq
 \|f(X) \| \leq \max_{U \in \cU_k} \| f(U) \|
\eeq
  for any $0<k\leq \min\{g',g\}$.  
  Thus if $f(X)=0$
  for all $X \in \Bgg(N)$ such that $X^{\T} X=I$ $($if $g'\geq g)$ or
  $XX^\T =I$ $($if $g'<g)$, then $f=0$.
  For example, if $g'\geq g$, then the set of isometries $\cU_g$ contains 
  the distinguished boundary of $\Bgg(N)$.
\end{prop}

\begin{proof}
 First suppose $f:\Bgg(N)\to \mathbb C$ (so that $d=1=d^\prime$).
 Pick any $U\in\cU_k$. By the maximum principle, the
 function $h(z)=f(zU)$ takes its maximum value 
 on $|z|=1$. 

 Now we use linear fractional automorphisms of the ball
 to prove that such an inequality holds for any $X$
 in the interior of $\Bgg(N)$.
 Select $\cF$ as in Lemma \ref{lemma:multilinfrac} which maps $0$ to $X$.
 Then $h(Z):= f( \cF(Z) )$ is analytic and maps $0$ to $f(X)$.
 The previous paragraph applies to give
 $$ 
   \|  f(X) \|  = \|  h(0) \|
    \leq  \max_{|z|=1}\| h( zU) \| =  
     \max_{|z|=1} \| f ( \cF( zU)) \|.
$$
 By Lemma \ref{lemma:multilinfrac}(1), $ \cF(zU) \in \cU_k $
  for $|z|=1$;
 so we have proved that the maximum of $f$ occurs on $\cU_k$.
 
 To prove the statement for matrix-valued $f$, simply note
 that given unit vectors $\gamma\in\mathbb C^d$ and $\eta\in\mathbb C^{d^\prime}$,
 the function $F(X)=\eta^\T  f(X)\gamma$ takes it maximum
 on $\cU_k$. It follows that
$$
  |F(X)| \le \max_{U\in \cU_k} \|f(U)\|.
$$
 Since $\gamma$ and $\eta$ are arbitrary, the result follows.
\end{proof}

\begin{rem}\rm
  This proposition has more content for larger $k$ and
  in particular $k=\min\{g,g^\prime\}$ is optimal.
\end{rem}

\subsection{Matrix Linksnullstellensatz}

For scalar NC analytic polynomials there is an elegant 
Linksnullstellensatz whose proof is due to Bergman, cf. \cite{HM}.
Now we generalize it to matrices with entries which are NC analytic
polynomials.

\begin{thm}\label{thm:Null}
Given a $m\times d$ matrix $P$ over $\CC\ax$
and a $n\times d$ matrix $Q$ over $\CC\ax$, suppose
that $P(X)v=0$
implies $Q(X)v=0$ for every matrix $g$-tuple $X$ and
vector $v$.
Then for some $G\in\CC\ax^{n\times m}$ we obtain
$Q=GP$.
\end{thm}

\begin{proof}
The rows of a matrix $A$ will be denoted by
$A_j = \left[
\begin{array}{cccc}a_{j1}& a_{j2}&\cdots & a_{jd}
\end{array}\right]$.
In particular, $P_j$ is a $1\times d$ matrix over $\CC\ax$.

Let $V_d=\CC\ax^{1\times d}$ denote the left $\CC\ax$-module
of $1\times d$ matrices of
polynomials. Note $P_j \in V_d$.
Let $I_d$ be the $\CC\ax$-submodule of $V_d$ generated by
the $P_j$, i.e.,
$$
I_d = \left\{\sum r_j P_j \mid r_j\in\CC\ax\right\}.
$$
$I_d$ is the smallest subspace of $V_d$ containing the $P_j$
and invariant
with respect to $M_j$=left multiplication by $x_j$ (for each $j$).

Note that $M_j$ determines a well defined linear mapping $Y_j$
on the quotient:
$$
  Y_j: V_d/I_d \to V_d/I_d.
$$

Let $W_{k}$ denote the image of polynomials
of degree at most $k$ in the quotient $V_d/I_d$.
These spaces are finite dimensional and $W_{k-1} \subseteq
W_k$. So $W_{k-1}$ is complemented in $W_k$.

Choose $N >\max{}$ degree of all polynomials in $P$ and $Q$.
Define $X_j = Y_j: W_{N-1}\to W_N$ and extend $X_j$ to a
linear mapping $W_N\to W_N$ in any
way (on a complementary subspace).

Let $v_j$ denote the element of $W_N$ determined by
the row with the polynomial 1 in the $j$-entry and $0$ elsewhere.
Define $v=\oplus v_j \in W_N^d$.

For a polynomial $q$, $q(X)v_j= \left[
\begin{array}{ccccccccc} 0& \cdots&0&q&0&\cdots&0
\end{array}\right]$ ($j$-th spot).
Hence
$Q_j(X)v=Q_j$. A similar statement is true for $P_j$; i.e.,
$P_j(X)v= P_j \in I_d$ and so $P_j(X)v=0$. So $Q_j(X)v=Q_j$ is
$0$ too which means $Q_j \in I_d$. Thus there exists $ G_{sj}$
such that
$$
  Q_j =\sum G_{js} P_s.
$$
Hence $Q=GP$, as desired.
\end{proof}

\section{The linear part of semi-distinguished ball maps}
 \label{sec:linear-bininding}

  Now that we have established preliminary results, 
we turn our attention to semi-distinguished ball maps, introduced in 
  \S \ref{subsec:distinguished}.
First we show that semi-distinguished ball maps have very distinctive 
linear parts. And then we set about to give properties of these
linear maps.

  A linear map $L:\CC^{g} \to \Mdd$ is a {\bf distinguished isometry}
  if it maps the {\it distinguished} boundary of $\cB_g$ to the
  boundary of $\Bdd$; i.e., if for
  each $X\in\cB_g$ with $X^\T X=I$ we have that $\|L(X)\|=1$.
    In this case a 
	  (nonzero) vector $\gamma$ such that
	    $\|L(X)\gamma\|=\|\gamma\|$ is called
		   a {\bf clinging vector} and this property
		     {\bf clinging}.

\begin{prop}\label{prop:linDist}  
 If $f$ is a semi-distinguished ball map, then
 $f^{(1)}$, the linear part of $f$, is a distinguished isometry.
\end{prop}

\begin{proof}
The proof is the same as that of Proposition \ref{prop:linIsom}.
\end{proof}

\subsection{Properties of distinguished isometries}

The remainder of this section is devoted to giving properties 
of distinguished isometries.

\begin{prop}
 \label{prop:cling} Let $L:\CC^{g}\to \Mdd$ be a linear
map.
\begin{enumerate}[\rm (1)]
\item
$L$ is a distinguished isometry
   if and only if 
\begin{equation}
  \label{nemesis}
     \Delta_{L}(X):= (X_1^\T X_1 +\cdots + X_g^\T X_g)\otimes I_d
        - L^\T (X)L(X) \succeq 0
 \end{equation} and clings $($i.e., $\Delta_L(X)$ is always positive
  semidefinite and never positive definite$)$.
\item
If $L$ is completely isometric,
   then it is a distinguished isometry. The converse is not true.
\end{enumerate}
\end{prop}

\begin{proof}
For the implication 
$(\Rightarrow)$ in
(1), given any $X_i$, choose a $W$ satisfying
$W^\T W=
X_1^\T X_1+\cdots +X_g^\T X_g$.
Note that it suffices to show \eqref{nemesis} on a dense
subset of $\cB_{g'}$. Thus we may assume that $W$ is invertible.
Then $(X_1W^{-1})^\T  X_1W^{-1} +\cdots + (X_gW^{-1})^\T  X_gW^{-1}
=
I$, so by assumption,
$I- L^\T (XW^{-1}) L(XW^{-1}) \succeq 0$ and it binds. Since $L$
is truly linear, we multiply this inequality with $W^\T $ on the
left and with $W$ on the right:
$W^\T W - L^\T (X)L(X)\succeq 0$ and it binds.
The converse $(\Leftarrow)$ is obvious.

First part of (2) is trivial. 
To finish the proof it suffices to exhibit an example
of a distinguished isometry which is not a complete
isometry.
Consider $L(x,y)=Ax+By$ with
$$
A=\left[\begin{array}{cccc}
1&0&0\\
0&\frac{\sqrt 2}2&0\\
0&0&0\\
0&0&\frac{\sqrt 2}2
\end{array}\right],
\quad
B=\left[\begin{array}{cccc}
0&0&0\\
\frac{\sqrt 2}2&0 &0\\
0&1&0\\
0&0&\frac{\sqrt 2}2
\end{array}\right].
$$

For $X=\left[\begin{array}{cc}
1&0\\
0&0\end{array}\right]$ and
$Y=\left[\begin{array}{cc}
0&0\\
1&0\end{array}\right]$,
$$
\left\|\left[\begin{array}{c}
X\\
Y\end{array}\right]\right\|=\sqrt 2
> \sqrt{\frac 32} =
\|L(X,Y)\|$$
This shows that $L$ is not a complete isometry.

It remains to be seen that $L$ satisfies \eqref{nemesis}.
We compute
$$
\Delta_L(x,y)=\left[\begin{array}{ccc}
\frac 12 y^\T y & -\frac{1}2 y^\T x & 0 \\
-\frac{1}2 x^\T y & \frac 12 x^\T x & 0 \\
0 & 0 & \frac 12 (x-y)^\T (x-y)
\end{array}\right].
$$
The top left $2\times 2$ block of $\Delta_L(x,y)$ can be factored
as
$$
\left[\begin{array}{cc}
-y^\T x^{- \T} & 1\\
1&0
\end{array}\right]
\left[\begin{array}{cc}
\frac 12 x^\T x & 0\\
0&0
\end{array}\right]
\left[\begin{array}{cc}
-x^{-1}y & 1\\
1&0
\end{array}\right].
$$
This immediately implies that for invertible $X$,
$\Delta_L(X,Y)$ is always positive semidefinite
and never positive definite. For noninvertible $X$ the same
holds true by a standard density argument.
\end{proof}

\begin{rem}\rm
  By way of contrast, every contractive $L:\CC^{g}\to\Mdd$
  is completely contractive. For related results see
  \S \ref{sec:further}.
\end{rem}

\subsubsection{The Gram representation}
 \label{sec:gram}

A powerful tool used is a matrix representation of a quadratic
NC polynomial. A key property of this representation is 
that matrix positivity of the quadratic NC polynomial
is equivalent to the positive semidefiniteness of the 
representing matrix.
The following lemma is needed to establish this.

\begin{lem}\label{lem:dense} For large enough $n$ the set
\beq\label{eq:denseSet}
\left\{ \VX w 
\mid
 X\in\left(\CC^{n\times n}\right)^g,\, 
\w\in\CC^{n}\right\}
\eeq
is all $\CC^{ng}$.\end{lem}

\begin{proof}
  Given $w,x_1,\dots,x_g \in \mathbb C^n$ with $x_j\ne 0$, choose $X_j\in\CC^{n\times n}$
  such that $X_j w =x_j$. For instance $X_j = x_j \frac{w^*}{\|w\|^2}$ will do.
\end{proof}

Note Lemma \ref{lem:dense} is true even with a \emph{fixed} $w\ne 0$ and
parametrizing over all $X$.

\begin{prop}\label{prop:gram}
Let
$$p=\sum_{1\leq i,j\leq g} x_i^\T  B_{ij} x_j$$
be a homogeneous quadratic NC polynomial with $B_{ij}\in\CC^{d'\times d}$.
Then there is a unique matrix $G\in 
(\CC^{d'\times d})^{g\times g}$ with
\beq\label{eq:gram1}
p=\Vx ^\T G\Vx
.
\eeq
Moreover,
$p(X)=\sum_{i,j}B_{ij}\otimes X_i^\T X_j$ is positive semidefinite for all $N\in\NN$ and
all
$X\in\big(\CC^{N\times N}\big)^g$ iff $G\succeq 0$.
\end{prop}

\begin{proof}
In $d'\times d$ block form, $G=\begin{bmatrix}B_{ij}
\end{bmatrix}_{i,j}$.
If $G\geq 0$, then $G=H^\T H$ for some matrix $H$. Hence
$p= \big( H\Vx  \big)^\T \big(H \Vx  \big)$ is a sum of hermitian
squares, so $p(X)\geq 0$ for all $N\in\NN$ and
$X\in\big(\CC^{N\times N}\big)^g$.
The converse follows from
Lemma \ref{lem:dense}.
\end{proof}

\subsubsection{Orthotropicity}

In this subsection we establish a basic property
of distinguished isometries $L$
 (that is, of those $L$ for which $\Delta_L$ is positive semidefinite and clinging), which we call orthotropicity.

A $d'\times d$ linear pencil $L=A_1x_1+\cdots+A_gx_g:\CC^g\to\Mdd$
is called {\bf orthotropic} if for every $X\in\CC^g$ and $w\in\CC^d$
satisfying $\|L(X)w\|=\|w\|$, the vector $L(X)w$ is orthogonal to the image of
$L(X^\perp)$.

\begin{prop}\label{prop:orthotropic}
Every distinguished isometry is orthotropic.
\end{prop}

To continue our analysis of distinguished isometries we write $L$ in
a special form.
We multiply $L$ with a unitary $V$
on the left and a unitary $U^\T$ on the right. Thus
without loss of generality, $A_1$ is the block matrix
\beq\label{eq:a1}
\left[\begin{array}{cc}
1&0\\
0&(A_1)_{22}
\end{array}
\right]
\eeq
and $A_j$ for $j\geq 2$ equals
\beq\label{eq:aj}
\left[\begin{array}{cc}
0 & (A_j)_{12}\\
(A_j)_{21}&(A_j)_{22}
\end{array}
\right] .
\eeq

\begin{proof}[Proof of Proposition {\rm \ref{prop:orthotropic}}]
Suppose $L=\sum_{i=1}^g A_ix_i$ is a distinguished
isometry and without loss of
generality write $L$ in the special form
described above.
Clearly,
orthotropicity of $L$ is equivalent to $(A_j)_{12}=0$ for $j\geq 2$.
In order to prove this we set all variables except for $X_1$, $X_j$ to $0$.
For convenience we use $X,Y$ (resp.~$x,y$) instead of $X_1,X_j$ (resp.~$x_1,x_j$) and
$A,B$
instead of $A_1,A_j$. Thus
$$L(x,y)= \left[\begin{array}{cc}
x& B_{12} (I_{d-1} \otimes y) \\
B_{21} y & A_{22} (I_{d-1} \otimes x) +B_{22} (I_{d-1} \otimes y)
\end{array}
\right].
$$
A straightforward computation shows
we can represent $\Delta_L(x,y)=x^\T x+y^\T y- L(x,y)^\T L(x,y)$ as
$\Vxy ^\T  G \Vxy$ (cf.~Proposition \ref{prop:gram}),
where $\Vxy$ stands for
$$
\Vxy =\left[\begin{array}{cc}
x&0\\ 0&I_{d-1} \otimes x\\ y&0 \\ 0&I_{d-1} \otimes y
\end{array}\right]
$$
and
$$G=\left[
\begin{array}{llll}
 0 & 0 & 0 & -B_{12} \\
0 &    I-A_{22}^\T  A_{22} & -A_{22}^\T  B_{21} &
-A_{22}^\T  B_{22}
   \\
 0 &
   -B_{21}^\T  A_{22} &
1-B_{21}^\T  B_{21} &
   -B_{21}^\T  B_{22} \\
  - B_{12}^\T  &
   -B_{22}^\T  A_{22} &
   -B_{22}^\T  B_{21} &
I   -B_{12}^\T  B_{12}-B_{22}^\T  B_{22}
\end{array}
\right].$$
If $\Delta_L(X,Y)$ is positive semidefinite for all $X,Y$, then
by Proposition \ref{prop:gram},
$G$ is positive semidefinite.
In particular, $B_{12}=0$.
(Note if $\Delta_L(X,Y)$ is only positive semidefinite for scalars $X,Y$, then
$B_{12}$ need not be $0$.) 

\noindent{\it Alternative proof of $B_{12}=0$.}
By density, we may assume $Y$ is invertible. $\Delta_L(X,Y)$
multiplied on the right by $\left[\begin{array}{cc}
Y^{-1}&0\\0&I\end{array}\right]$ and on the left by the transpose of the
same matrix yields
$M_2:=\left[
\begin{array}{cc}
m_{11}&m_{12}\\m_{21}&m_{22}\end{array}\right]
$
for
\begin{eqnarray}\label{eq:normNem}
m_{11}&=&
1-B_{21}^\T B_{21}\nonumber\\
m_{12} &=
 & - B_{21}^\T A_{22}( I_{d-1} \otimes X)-
B_{21}^\T  B_{22} (I_{d-1} \otimes  Y) + Y^{-\T} X^\T  B_{12} (I_{d-1} \otimes
 Y)
\nonumber\\
m_{21} &=& m_{12}^\T =
- ( I_{d-1} \otimes X^\T ) A_{22}^\T  B_{21} - (I_{d-1} \otimes Y^\T ) B_{22}^\T  B_{21} + (I_{d-1} \otimes Y^\T ) B_{12}^\T  X Y^{-1} \\
m_{22}&=&
I_{d-1} \otimes (X^\T X+ Y^\T Y)
- (I_{d-1} \otimes Y^\T )B_{12}^\T B_{12} (I_{d-1} \otimes Y)-\nonumber\\
&&\quad -\big((I_{d-1} \otimes X^\T )A_{22}^\T +(I_{d-1} \otimes Y^\T )B_{22}^\T \big)\big(A_{22} (I_{d-1} \otimes X)+B_{22}(I_{d-1} \otimes Y)\big).\nonumber
\end{eqnarray}
Consider $m_{12}$ and note that
$$
Y^{-\T} X^\T  B_{12} (I_{d-1} \otimes Y)= Y^{-\T} X^\T  Y B_{12}.
$$
Suppose $B_{12}\ne 0$.
Then it is easy to construct $X=X(\eps)$, $Y=Y(\eps)$ sending
this term to $\infty$ as $\eps \to 0$ while keeping all the
remaining terms bounded. This contradiction yields $B_{12}=0$.
\end{proof}

\section{Characterization of semi-distinguished ball maps}
 \label{sec:mostly-linear}

The following theorem summarizes what we know about semi-distinguished
pencil ball maps. Both the hypotheses and the conclusions
are weaker than those of Theorem \ref{thm:ball}.
The relationship between both results 
is made precise by Corollary \ref{cor:oldBallDealThm} of
this section.

\begin{thm}\label{thm:bidisk}
Let $L$ be a $d'\times d$ NC analytic truly linear
pencil and $f:
\cB_{g'}\to\cB_{L}$ a semi-distinguished pencil ball map with $f(0)=0$.
Clearly, $h:=L\circ f$ maps $\cB_{g'}\to\cB_{d'\times d}$. Write $h$
as $h=h^{(1)}+h^{(\infty)}$, where $h^{(1)}$ is
the linear homogeneous component in the NC power series expansion
of $h$ and $h^{(\infty)}=\sum_{\alpha=2}^\infty h^{(\alpha)}$.
Then there is a unique contraction $M\in \big(\CC^{d'\times d}\big)^{g'}$
and a unique nontrivial subspace $\cS\subseteq\CC^{dg'}$ such that:
\ben[\rm (1)]
\item
$h^{(1)}(x)=M \Vx $, $M|_{\cS}$ is an isometry and
$M\Pi_{\cS^\perp}$ is a strict contraction.
\item
Each $h^{(\alpha)}(x)$ for $\alpha\geq 2$ is of the form
$P_\alpha \Pi_{\cS^\perp} \Vx $ for a matrix $P_\alpha$ of NC polynomials.
\item
For the formal NC power series $P^{(\infty)} := \sum_{\alpha = 2}^\infty P_\alpha$,
$P^{(\infty)}(X)v:=\sum_{\alpha=2}^\infty P_\alpha (X) v$ converges for
$v\in \cS^\perp \otimes
\CC^N$ and $X\in \left(\CC^{N\times N}\right)^g$ in an NC $\eps$-neighborhood of $0$.
Also, $h^{(\infty)}(x)=P^{(\infty)}(x) \Pi_{\cS^\perp} \Vx $.
\item
$\left\|\big(M\otimes I_N+P^{(\infty)}(X)\big) \Pi_{\cS^\perp\otimes\CC^N}\right\|\leq 1$ for $X\in \left(\CC^{N\times N}\right)^g$ with $\|X\|< 1$ and $(M\otimes I_N)(\cS\otimes\CC^N)$
is orthogonal to
$(M\otimes I_N)(\cS ^\perp \otimes \CC^N)$, to $P_\alpha(X)(\cS ^\perp)$ for all $\alpha\geq 2$
and to $P^{(\infty)}(X) (\cS ^\perp)$.
\een
\end{thm}

\begin{proof}
Let $\Delta_{h^{(1)}}(x):=x^\T x-h^{(1)}(x)^\T h^{(1)}(x)=\Vx ^\T G\Vx $
be as in Proposition {\rm \ref{prop:gram}},
where $G \geq 0$. Write $h^{(1)}(x)=M \Vx $ and
note that $G=1-M^\T M$.

Let $\cS:=\ker G = \ker (I-M^\T M)=
\range (I-M^\T M)^\perp $.
By the clinging property, $\cS\neq\{0\}$.
By definition, $M|_{\cS}$ is an isometry. Conversely, if $v$ satisfies
$\|Mv\|=\|v\|$, then $\langle Mv,Mv\rangle=\langle v,v\rangle$ and hence
$\langle v,(I-M^\T M)v\rangle =0$. Since $M$ is a contraction,
$I-M^\T M$ is positive semidefinite. Thus $(I-M^\T M)v=0$, that is,
$v\in\cS$. This proves (1) and also the uniqueness of $M$ and $\cS$.

 For (2) fix $N\geq 1$ and let 
 $X\in\cB_{g^\prime}(N)$ such that $\|X \|=1$ be given.  By 
 equation \eqref{eq:schwarzBinds} of Schwarz's lemma 
 (Theorem \ref{thm:schwarz-lemma}) applied to
 $h(zX)$, $|z|<1$ 
 for all $0\leq\delta<1$
and $0\leq \theta \leq 2\pi$ we have
\begin{equation}\label{eq:hIneq1}
 0\preceq \delta^2 X^\T X -h(\delta e^{i\theta}X)^\T  h(\delta e^{i\theta}X).
\end{equation}
If $\delta$ is in
the series radius, we may write $h(\delta e^{i\theta}X)=
h^{(1)}(\delta e^{i\theta}X)+ h^{(\infty)}(\delta e^{i\theta}X) =
\sum_{\alpha=1}^\infty h^{(\alpha)}(\delta e^{i\theta}X)$.
We integrate \eqref{eq:hIneq1} to obtain
\begin{equation}\begin{split}\label{eq:hIneq0} 0 &\preceq
\frac 1{2\pi}
\int^{2\pi}_0\big(\delta^2X^\T X-h(\delta e^{i\theta}X)^\T  h(\delta e^{i\theta}X)\big)
d\theta
\\ &=
\delta^2X^\T X-\delta^2h^{(1)}(X)^\T h^{(1)}(X) - \frac 1{2\pi}
\int^{2\pi}_0 h^{(\infty)}(\delta e^{i\theta} X)^\T  h^{(\infty)}(\delta e^{i\theta}X) d\theta \\
&=
\delta^2 X^\T X-\delta^2h^{(1)}(X)^\T h^{(1)}(X) - \sum_{\alpha=2}^\infty h^{(\alpha)}(\delta X)^\T  h^{(\alpha)}(\delta X)
.
\end{split}
\end{equation}

By Proposition \ref{prop:cling}, $\Delta_{h^{(1)}}(X)\succeq 0$ with clinging. Thus
by \eqref{eq:hIneq0}, for every $w$ satisfying
\begin{equation}\label{eq:hIneq2}
\big(X^\T X-h^{(1)}(X)^\T h^{(1)}(X)\big) w=0
\end{equation}
we have $h^{(\alpha)}(\delta X)w=0$ for $\alpha\geq 2$ and
$\delta$ in the series radius.
In particular, by Proposition \ref{prop:gram},
\eqref{eq:hIneq2} is equivalent to
$\sqrt G\Vx  w=0$ and this implies that $h^{(\alpha)} (X)w=0$ for $\alpha\geq 2$ and
every $X$ in the series radius.
By a scaling argument ($h^{(\alpha)}$ is homogeneous),
the same holds true for every $X$ and $w$.
Hence the matrix NC Nullstellensatz Theorem \ref{thm:Null} applies and implies that there is a matrix
of NC polynomials $\tilde P_\alpha$ with $\tilde P_\alpha (x) \sqrt G
\Vx=h^{(\alpha)} (x)$.
Since $\sqrt G=  \sqrt G (\Pi_{\cS}+\Pi_{\cS^\perp})=
 \sqrt G\, \Pi_{\cS^\perp}$, we set $P_\alpha=\tilde P_\alpha
 \sqrt G$. Then $h^{(\alpha)}(x)=P_\alpha (x) \Pi_{\cS^\perp} \Vx $.

(3) The second part is clear and for the first statement we refer
the reader to \cite{KVV}.

(4) Let $v\in\cS$ and $w\in\cS^\perp$. Then
\begin{equation}
\langle Mv,Mw\rangle = \langle M^\T Mv,w\rangle =
\langle v,w\rangle=0
\end{equation}
since $(1-M^\T M)v=0$. This shows that $M(\cS^\perp)$ is orthogonal to
$M(\cS)$. For the strengthening with tensor products, let $s_i\in\cS$,
$t_i\in\cS^\perp$, $v_i,u_j\in\CC^N$. Then
\begin{equation*}
\begin{split}
\langle (M\otimes I_N)(\sum_i s_i\otimes v_i), (M\otimes I_N)(\sum_j
t_j\otimes u_j)  \rangle &= \langle \sum_i (Ms_i \otimes v_i), \sum_j
(Mt_j \otimes u_j)  \rangle \\
&= \sum_{i,j} \langle Ms_i,Mt_j\rangle
\langle v_i,u_j\rangle=0.
\end{split}
\end{equation*}

Let $h(x)=\tilde h(x) \Vx$, where
$$\tilde h(x)=M+\sum_{\alpha\geq 2}P_\alpha  (x) \Pi_{\cS^\perp}.$$
By Theorem \ref{thm:scott} (applied with $H=\tilde h$), $\|\tilde h(X)\|\leq 1$ for
all $X$ with $\|X\|< 1$.

Rewrite $\tilde h(x)$ as
\beq\label{eq:2parts}
\tilde h(x)=M\Pi_{\cS}+ (M+ \sum_{\alpha\geq 2}P_\alpha (x) )\Pi_{\cS^\perp}.
\eeq
Both summands have norm $\leq 1$ for $X$ with $\|X\|< 1$.
In particular,
$$\|
 (M\otimes I_N+ \sum_{\alpha\geq 2}P_\alpha (X) )\Pi_{\cS^\perp\otimes\CC^N}
\|\leq 1,
$$
as desired.

Clearly, $\tilde h(x)|_{\cS}=M|_{\cS}$
is an isometry and thus $\tilde h(X)(\cS\otimes\CC^N)$ is orthogonal to
$\tilde h(X)(\cS^\perp\otimes\CC^N) = (M\otimes I_N+P^{(\infty)}(X))(\cS^\perp
\otimes\CC^N)$. Since $(M\otimes I_N)(\cS\otimes\CC^N)$ is
orthogonal to $(M\otimes I_N)(\cS^\perp\otimes\CC^N)$, this implies
$(M\otimes I_N)(\cS^\perp\otimes\CC^N) \perp P^{(\infty)}(X)
(\cS^\perp\otimes\CC^N)$.

Suppose $w \in (M\otimes I_N)(\cS^\perp\otimes\CC^N)$. Then
$w^\T  P^{(\infty)}(tX)
(\cS^\perp\otimes\CC^N) =\{0\} $ for small enough $t$.
Then
$$0= w^\T P^{(\infty)}(tX)= w^\T \sum_{\alpha\geq 2} t^\alpha P_\alpha(X)
= \sum_{\alpha\geq 2} t^\alpha (w^\T P_\alpha(X)) $$
implies $(M\otimes I_N)(\cS^\perp\otimes\CC^N) \perp P_\alpha (X)
(\cS^\perp\otimes\CC^N)$.
\end{proof}

Let us note in passing that under the conditions of the previous theorem,
\eqref{eq:hIneq2}
implies $h^{(\infty)}(X)w=0$ for all $X\in\cB_g$. Indeed, let us consider
the analytic function $z\mapsto h^{(\infty)}(z X)w$ on $\DD$.
Clearly, \eqref{eq:hIneq2} holds for $X$ replaced by $\delta X$ due to
homogeneity. If $\delta$ is in the series radius, then
by the NC power series expansion and the lemma,
$$ h^{(\infty)}(\delta X)w
= \sum_{\alpha=2}^\infty  h^{(\alpha)}(\delta X)w=0.$$
Thus by analytic continuation,
$ h^{(\infty)}(X)w=0$.

Next we give a corollary which makes the relationship
between Theorem \ref{thm:bidisk} 
and Theorem \ref{thm:ball} clearer.

\begin{cor}\label{cor:oldBallDealThm}
Keep the assumptions and notation from Theorem {\rm \ref{thm:bidisk}}.
If, in addition, $f$ is a pencil ball map, then $M$ is a complete
isometry.

Conversely, $h=L\circ f$ satisfying
{\rm (1), (2), (3), (4)} for a complete isometry $M$
is an NC ball map $\cB_{g'}\to\cB_{d'\times d}$
sending $0$ to $0$.
\end{cor}

\begin{proof}
Suppose $f$ is a pencil ball map. Then
$h^{(1)}$ is a (linear) NC ball
map by Proposition \ref{prop:linIsom}. Hence $h^{(1)}=M\Vx $ with $M$
a complete isometry (see Theorem \ref{thm:complete-iso-structure}).

For the converse, suppose $h$ satisfies (1)--(4). By (1) and (3),
$h(x)=\tilde h(x) \Vx$,
where $\tilde h(x)$ is given by:
$$
\tilde h(x)=M\Pi_{\cS}+ (M+ \sum_{\alpha\geq 2}P_\alpha (x) )\Pi_{\cS^\perp}.
$$
(1) and (4) implies that $\tilde h(X)$ is for $X\in\cB_{g'}$ an
orthogonal sum of two contractions,
thus $\tilde h(X)$ is a contraction for $X\in\cB_{g'}$, i.e., $\|\VX\|\leq 1$.
Hence $h(X)=\tilde h(X)\VX$ is a contraction.

For the binding property of $h$ we use that $M$ is a complete
isometry.
Let $e$ denote the distinguished vector associated with $M$, that is,
$A_j^\T A_ie=\delta_i^je$ if $M=\left[\begin{array}{ccccc}
A_1&\cdots&A_{g'}\end{array}\right]$ (cf. Proposition \ref{prop:ball-map}).
By (1),
$h^{(1)}(X)$ binds
at $e\otimes w$, where $w$ is a binding vector for $I-X^\T X$, i.e.,
$X^\T Xw=w$. This concludes the proof since
$h(X) (e\otimes w)= h^{(1)}(X)(e\otimes w)$
by (2) and (3).
\end{proof}

\section{Further analysis of distinguished isometries}
 \label{sec:further}
We have successfully classified complete isometries, see
Theorem \ref{thm:complete-iso-structure}.
Distinguished isometries are more challenging and a few sample results
are provided below.

\begin{thm}\label{thm:2var}
Suppose $L$ is an orthotropic linear pencil in $2$ variables.
If $\Delta_L(X_1,X_2)\succeq 0$ for all  
$X_1,X_2\in\CC^{n\times n}$,
and clings for all {\it scalar} $X_1,X_2\in\CC$, then
$\Delta_L(X_1,X_2)$ clings for all
$X_1,X_2\in\CC^{n\times n}$.
\end{thm}

\begin{rem}
We conjecture based on inconclusive computer experiments that
Theorem \ref{thm:2var} is false in 3 variables.
\end{rem}

\subsection{Equations which reformulate the clinging property} 

Throughout this subsection $L$ will denote an orthotropic $d'\times d$
linear pencil in $g$ variables.
We assume it clings for $X\in\CC^g$.
Let $$\Delta_L(x)=x^\T x-L(x)^\T L(x)=\Vx ^\T G \Vx $$ be the Gram
representation as in Proposition \ref{prop:gram}.
Given $G\in(\CC^{d\times d})^{g\times g}$ 
we call the linear subspace of its
kernel spanned by all the vectors of the form
$$\left[\begin{array}{c}
\al_{1} v \\
\vdots \\
\al_{g} v
\end{array}\right]\in\ker G
$$
the {\bf scalar binding kernel} $\cN_0$. 
(Since $L$ clings for $X\in\CC^g$,
for every $\al_1,\ldots,\al_g\in\CC$ there exists a $0\ne v\in\CC^d$ with
$\left[\begin{array}{c}
\al_{1} v \\
\vdots \\
\al_{g} v
\end{array}\right]\in\ker G$.)

Fix a basis
$$
\left\{
\eta_i:=
\left[\begin{array}{c}
\al_{i,1} v_i \\
\vdots \\
\al_{i,g} v_i
\end{array}\right]
\mid i=1,\ldots, t+m\right\} \subseteq \CC^{gd}$$
for the scalar binding kernel $\cN_0$ of $G$.
We assume that $\{v_1,\ldots, v_t\}$ is a maximal linearly independent set and
that
$$
v_{t+j}= \sum_{i=1}^t  \gamma_{ji} v_i
$$
for $j=1,\ldots,m$.

Let $X_1,\ldots,X_g\in\CC^{n\times n}$. We assume that $X_1$ is invertible
and define $Z_i:=X_1^{-1}X_i$.
(Matrix) binding at $X$ is equivalent to the existence 
(for all $Z_i$)
of a nontrivial solution
to $\Delta_L(I_n,Z_2,\ldots,Z_g)v=0$. This is implied by the existence of $r_i\in\CC^n$ for
which there
is a nonzero $v\in\CC^{dn}$ such that
\begin{equation}
\label{eq:mtxBind}
\left[\begin{array}{c}
(I_d \otimes I_n) v \\
(I_d\otimes Z_2) v \\
\vdots \\
(I_d\otimes  Z_g) v
\end{array}\right]
= \sum_{i=1}^{t+m} \eta_i \otimes r_i.
\end{equation}
In particular,
\begin{eqnarray*}
v&=& \sum_{i=1}^{t+m} v_i \otimes r_i= \sum_{i=1}^t  v_i\otimes r_i
+ \sum_{j=t+1}^{t+m} \sum_{i=1}^t  \gamma_{ji} v_i \otimes r_j= \\
&=& \sum_{i=1}^t  v_i \otimes \underbrace{(r_i + \sum_{j=t+1}^{t+m} \gamma_{ji} r_j )}_{\Gamma_i(r)}
\end{eqnarray*}
Similarly, $(I_d\otimes Z_k)v= \sum_{i=1}^t  v_i \otimes \Gamma_i(Z_k r)$.
Using this in \eqref{eq:mtxBind} yields
$$
\sum_{i=1}^t  v_i\otimes \Gamma_i (Z_k r) = \sum_{i=1}^t
v_i\otimes \Gamma_i\big(r \diag(\alpha_k) \big),
$$
where $r=\left[\begin{array}{ccc} r_1 & \cdots & r_{t+m}\end{array}\right]$
and $\diag(\alpha_k)$ is the diagonal matrix with $\alpha_{i,k}$ as its $(i,i)$ entry.
Linear independence of the $v_1,\ldots,v_t$ gives
$\Gamma_i\big(Z_kr-r \diag(\alpha_k)\big)=0$ for all $k=2,\ldots, g$ and
$i=1,\ldots, t$.
Thus for all these $i,k$:
\begin{equation}
\label{eq:bindSys}
(Z_k-\alpha_{i,k})r_i+ \sum_{j=t+1}^{t+m} \gamma_{ji}(Z_k-\alpha_{j,k})r_j=0.
\end{equation}
Hence if all the $Z_k-\alpha_{i,k}$ are invertible,
\begin{equation}\label{eq:mtxBindEq}
r_i=-\sum_{j=t+1}^{t+m} b_{ij}\big(\diag(\alpha_k),Z_k\big) r_j,
\end{equation}
where
$$
b_{ij}\big(\diag(\alpha_k),Z\big) := \gamma_{ji} (Z-\alpha_{i,k})^{-1}  (Z-\alpha_{j,k}).
$$

Equations derived so far reformulate 
then clinging property and we say how precisely in the following
lemma.

\begin{lem}\label{lem:key1}
Consider the following conditions:
\ben[\rm (i)]
\item
For each $Z_2,\ldots, Z_g$ the system of equations \eqref{eq:bindSys}
has a solution $r_1,\ldots, r_{t+m}$;
\item
$\Delta_L$ clings.
\een
Then {\rm (i)} $\Rightarrow$ {\rm (ii)} and if $\cN_0=\ker G$, then
{\rm (ii)} $\Rightarrow$ {\rm (i)}.
\end{lem}

\begin{proof}
Follows from the computations given above.
\end{proof}

\subsection{The general case and the proof of Theorem \ref{thm:2var}}

Now we give a theorem more general than Theorem \ref{thm:2var}
that implies Theorem \ref{thm:2var}.

\begin{thm}
\label{thm:bindIneq}
Suppose $t(g-2) < m$. If $\Delta_L(X)\succeq 0$ for all $X\in\big(\CC^{n\times n}\big)^g$
and clings for all {\it scalar} $X\in\CC^g$, then
$\Delta_L(X)$ clings for all
$X\in\big(\CC^{n\times n}\big)^g$.
\end{thm}

\begin{proof}
We assume all $b_{ij}(\diag(\alpha_k),Z_k)$ exists, i.e.,
all $Z_k-\alpha_{i,k}$ are invertible. This causes no
loss of generality: the set of all matrix $g$-tuples that make $\Delta_L$ cling
is closed and our condition implies clinging on a dense subset.

Equation \eqref{eq:mtxBindEq} gives $r_i=r_i(k)$ as a function of $k$.
By Lemma \ref{lem:key1} we need to show that for every choice of $Z_i$
the system \eqref{eq:bindSys} has a solution, i.e.,
$r_i(2)=r_i(3)=\cdots =r_i(g)$ for all
$i=1,\ldots, t$. This yields $tn(g-2)$ homogeneous equations in $mn$ unknowns. Thus
if $m>t(g-2)$ this system will always have a nontrivial solution.
\end{proof}

\begin{proof}[Proof of Theorem {\rm \ref{thm:2var}}]
Fix a basis
$$
\left\{
\left[\begin{array}{c}
\al_i v_i \\
\beta_i v_i
\end{array}\right]
\mid i=1,\ldots, t+m\right\}$$
for the scalar binding kernel $\cN_0$,
where $\{v_1,\ldots,v_t\}$ is a maximal linearly independent set.
In view of Theorem \ref{thm:bindIneq} it suffices to show that $m>0$.

Suppose $m=0$ and choose $\al,\be$ with $\frac{\al}{\be} \neq \frac{\al_i}{\be_i}$
for all $i$. By scalar binding, there is a nonzero vector $u$ with
$\left[\begin{array}{c}
\al_i u \\
\beta_i u
\end{array}\right]
\in \cN_0$, i.e., for some $\lambda_i$:
$$
\left[\begin{array}{c}
\al u \\
\beta u
\end{array}\right]
=\sum_{i=1}^t  \lambda_i \eta_i =\sum_{i=1}^t  \lambda_i \left[\begin{array}{c}
\al_i v_i \\
\beta_i v_i
\end{array}\right]
$$
Hence
\begin{eqnarray*}
\be \sum_{i=1}^t  \lambda_i \al_i v_i &=& \al  \sum_{i=1}^t  \lambda_i \be_i v_i
\end{eqnarray*}
and thus by the linear independence of the $v_i$,
$\be \lambda_i \al_i= \al \lambda_i \be_i$ for all $i=1,\ldots, t$.
As at least one $\lambda_j$ is nonzero, this implies
$$
\frac{\al}{\be}=\frac{\al_j}{\be_j}
,
$$
contrary to our assumption. Thus $m>0$, as desired.
\end{proof}

\section{Acknowledgments}
The authors thank Victor Vinnikov and 
Dima Kalyuzhnyi-Verbovetski\u\i{} for helping us with NC analytic functions.


\newpage

\centerline{NOT FOR PUBLICATION}

\tableofcontents

\end{document}